\documentclass[reqno, 11pt, a4paper]{amsart}
\usepackage{amsmath, amssymb}  
\usepackage{amsthm}
\usepackage[foot]{amsaddr}
\usepackage{mathtools}
\usepackage{mathrsfs, enumerate}
\usepackage[truedimen,top=3truecm, bottom=3truecm, left=2.5truecm, right=2.5truecm, includefoot]{geometry} 
\begin{document}

\title[A Stefan problem with non-linear diffusion from particle systems]{A one-phase Stefan problem with non-linear diffusion from highly competing two-species particle systems}
\author[K. Hayashi]{Kohei Hayashi}
\address{Graduate School of Mathematical Sciences, The University of Tokyo, Komaba, Tokyo 153-8914, Japan.}
\email{kohei@ms.u-tokyo.ac.jp}
\keywords{Hydrodynamics limit, Interacting particle system, Fast-reaction limit, Stefan problem}
\subjclass[2010]{60K35, 82C22, 35K57, 35B40}
\maketitle

\newtheorem{definition}{Definition}[section]
\newtheorem{theorem}{Theorem}[section] 
\newtheorem{lemma}[theorem]{Lemma}
\newtheorem{proposition}[theorem]{Proposition}
\newtheorem{remark}{Remark}[section]
\newtheorem{assumption}{Assumption}[section]

\makeatletter
\renewcommand{\theequation}{%
\thesection.\arabic{equation}}
\@addtoreset{equation}{section}
\makeatother


\newcounter{num}
\newcommand{\Rnum}[1]{\setcounter{num}{#1}\Roman{num}}

\begin{abstract}
We consider an interacting particle system with two species under strong competition dynamics between the two species. Then, through the hydrodynamic limit procedure for the microscopic model, we derive a one-phase Stefan type free boundary problem with non-linear diffusion, letting the competition rate divergent. Non-linearity of diffusion comes from a zero-range dynamics for one species while we impose the other species to weakly diffuse according to the Kawasaki dynamics for technical reasons, which macroscopically corresponds to the vanishing viscosity method.  
\end{abstract}

\section{Introduction}
Consider two types of species under high competition. As the reaction rate between the two species tends to infinity, they would be spatially segregated and interface appears as boundary of their moving territories. As a mathematical problem, such a limiting procedure is called the spatial-segregation limit and was originally formulated in \cite{HHP96} and \cite{DHMP99} as follows. For any positive number $K $, consider a pair of non-negative solutions $(u_K, v_K) $ of the reaction-diffusion system 
\begin{equation}
\label{RD}
\begin{cases}
\begin{aligned}
& \partial_t u = d_1 \Delta u - K u v \\
& \partial_t v = d_2 \Delta v - K u v 
\end{aligned}
\end{cases}
\end{equation}
with some additional constraints (usually with Neumann boundary condition on a bounded domain). The spatial dimension is arbitrary in the sequel. Here $ d_1 > 0$ and $d_2 \ge 0$ are diffusion coefficients. As the competition rate $K$ tends to infinity, solutions $(u_K, v_K)$ converges to some $( u, v ) $ which satisfies $u v \equiv 0$, in addition, $u$ and $v$ satisfies heat equations with diffusion coefficient $d_1$ and $d_2$, respectively, on their moving domains. In other words, the solutions of \eqref{RD} are spatially segregated, asymptotically as the competition rate tends to infinity, and interface governed by a free boundary problem appears. Alternatively, since the above formulation is concerned with a singular limit for reaction terms, it is also called the \textit{fast-reaction limit} in some literature. In \cite{HHP96}, the authors treated the case with $d_2 = 0$ and derived a free boundary problem called a one-phase Stefan problem. After that, \cite{DHMP99} proved that when $d_2 > 0$, a pair of solutions $(u_K, v_K) $ converges to a solution of a two-phase Stefan problem under Neumann boundary condition and \cite{CDHMN04} treated the problem under inhomogeneous Dirichlet boundary condition. All the above results consider the singular limit problem for balanced reaction terms, namely when the reaction terms of the reaction-diffusion system coincide up to some constant. For the singular limit problem under unbalanced reaction, \cite{IMMN17} studied the case when $d_2 = 0 $ and the reaction terms are given by monomials with various exponents. They derived three kinds of liming behavior as the competition rates tend to infinity. At present, few results for unbalanced reaction are known other than this. To other direction, \cite{HHP00} studied the spatial-segregation limit with a presence of non-linear diffusion. They considered the reaction-diffusion system of the form
\begin{equation}
\label{NLRD}
\begin{cases}
\begin{aligned}
& \partial_t u = \Delta \varphi (u ) - K u v \\
& \partial_t v =  - K u v 
\end{aligned}
\end{cases}
\end{equation}
under an inhomogeneous Dirichlet boundary condition and studied the limiting behavior as the reaction rate $K$ tends to infinity. Here $\varphi$ is an increasing function. The limiting equation becomes the following one-phase Stefan problem. 
\begin{equation}
\begin{cases}
\begin{aligned}
\label{aim}
& \partial_t u = \Delta \varphi (u) && \text{ in } \{ u > 0 \}, \\
& \partial_t v = v |_{ t = 0 } && \text{ in } \{ v > 0 \} .
\end{aligned}
\end{cases}
\end{equation}
The paper \cite{HHP00} introduced the notion of weak solutions in order that the limiting problem makes sense allowing even degenerate diffusion, namely when the derivative of the increasing function $ \varphi $ degenerates at zero.

Recently, \cite{DMFPV19} considered an interacting particle system as a microscopic counterpart of the above high competition dynamics corresponding to the system \eqref{RD} with $d_1, d_2 > 0 $ and derived the two-phase free boundary problem by which the asymptotic interface is governed, directly from the microscopic model through a scaling limit called the hydrodynamic limit. They considered two-species exclusion processes where the same kind of particles can not stay at the same site and two different kinds of particles coexisting at the same site are killed with rate $K (N)$, which diverges as the scaling parameter $N$ tends to infinity. For the case where competition between two species is unbalanced, \cite{H20} considered two-species exclusion processes which macroscopically correspond to the reaction-diffusion systems studied in \cite{IMMN17} and derived the three kinds of interface behavior directly from the particle systems through the hydrodynamic limit. In the model of \cite{H20}, one species has no diffusion dynamics even microscopically, which asserts the diffusion coefficient of macroscopic density, say $v $, of that kind of particles exactly equals to zero. For another kind of research concerning singular limit for interacting particle systems, \cite{FT18} considered a one-species exclusion process with creation and annihilation dynamics whose hydrodynamic limit equation becomes the Cahn-Hilliard equation with a double-well potential (Allen-Cahn type equation) and they studied a fast-reaction limit problem when the reaction rate diverges depending on the scaling parameter of the microscopic model. Among all these microscopic models, the diffusion dynamics is commonly the Kawasaki process, that is, the dynamics of many independent simple random walks under an exclusive restriction, which deduces linear diffusion in hydrodynamic equations.

The aim of this paper is to derive a Stefan type moving boundary problem with possibly non-linear diffusion \eqref{aim} by scaling limits for microscopic systems. For that purpose, we consider a zero-range process which admit the arbitrary finite number of particles to exist at one site, for which each particle jumps to neighboring sites with a rate depending only on the number of particles at that site. Then one can show that there exists a function $\varphi $ which is naturally determined by jump rates of microscopic particles satisfying some conditions, especially a linear growth condition (see \textbf{(LG)} below) for technical reasons. In this paper, we consider a two-species model on the general dimensional torus with competition which macroscopically corresponds to the reaction-diffusion system \eqref{NLRD}, as in \cite{HHP00}, to derive the problem \eqref{aim} by letting the competition rate to be divergent asymptotically as the scaling parameter tends to infinity. As a microscopic counterpart of this one-phase model, it is natural to consider interacting particle systems with two components, say type-$1$ and type-$2$ particles, where only type-$1$ particles diffuse according to the zero-range dynamics while type-$2$ particles do not move at all. However, treatment of such a situation is technically difficult since for two-species systems, spectral gap estimates, which play an important role to prove hydrodynamic limits, fail to hold unless both types of particles diffuse. To overcome this and prove the hydrodynamic limit in our model, we consider weak diffusion dynamics for type-$2$ particles. Moreover, since the non-compactness of configuration spaces makes the problem more technically difficult, we impose an exclusion rule to type-$2$ particles and let them diffuse according to Kawasaki dynamics with asymptotically negligible diffusion. This macroscopically corresponds to the vanishing viscosity method in PDE theory as follows. We let $\varepsilon $ be the diffusion coefficient of type-$2$ density. If $\varepsilon $ and $K$ is independent of the scaling parameter, then the system of hydrodynamic limit equations becomes 
\begin{equation*}
\begin{cases}
\begin{aligned}
&\partial_t u  = \Delta \varphi( u ) - K u  v \\
&\partial_t v = \varepsilon \Delta v - K u v  
\end{aligned}
\end{cases}
\end{equation*}
for some non-negative increasing function $\varphi $ determined by the zero-range jump rate of type-$1$ particles. Here $u$ and $v$ are macroscopic density of type-$1$ and type-$2$ particles, respectively. Taking the limit as $K \to \infty$ and $\varepsilon \to 0$ in the above system, we can derive a quasi-linear free-boundary problem as in \cite{HHP00} noting the existence of a viscosity term does not affect the limiting problem when the reaction rate tends to infinity. What we consider in this paper is to assume $\varepsilon = \varepsilon (N) $ vanishes and $K = K (N) $ diverges as the scaling parameter $N$ tends to infinity and derive the limiting free-boundary problem directly from our microscopic model. 

Finally, we summarize the organization of this paper here. In Section \ref{sec:model}, we give the precise description of our microscopic model and state the main result. Our main result is to derive a one-phase free-boundary problem with non-linear diffusion \eqref{aim} directly from the zero-range-Kawasaki process with existence of singular competition and slow diffusion for type-$2$ particles. In Section \ref{sec:main thm}, we give a strategy to prove the main theorem stated in Section \ref{sec:model}. The proof is based on Yau's relative entropy method which is introduced in \cite{Yau91}. Then in Section \ref{sec:prob}, we show that our microscopic density profile is close to macroscopic densities whose time evolution is governed by a semi-discretized system of our hydrodynamic limit equation. This probabilistic procedure is crucial for the proof of the main theorem and it is shown assuming the multi-variable Boltzmann-Gibbs principle. In Section \ref{sec:PDE estimates}, we give some a priori estimates for the semi-discretized system. Then in Section \ref{sec:BG}, we give the proof of our Boltzmann-Gibbs principle for multi-variable functions. Finally in Section \ref{sec:PDE}, we show a convergence result for the semi-discretized macroscopic system to the Stefan problem with non-linear diffusion \eqref{aim}, which completes the proof of the main result combining with the probabilistic part given in Section \ref{sec:prob}.

\section{Model and main result}
\label{sec:model}
\subsection{Microscopic model}
Throughout this paper, we write $\mathbb{N} = \{ 1, 2, \ldots \} $, $\mathbb{Z}_+ = \{ 0, 1, \ldots \}$ and $\mathbb{R}_+ = [ 0, \infty ) $. Let $\mathbb{T}^d_N \cong \{ 1, 2,...,N \}^d$ and $\mathbb{T}^d \cong [ 0, 1)^d  $ be $d$-dimensional discrete and continuous tori and let $\mathcal{X}_N = \mathcal{X}^1_{  N } \times \mathcal{X}^2_{ N  } \coloneqq \mathbb{Z}_+^{\mathbb{T}^d_N} \times \{ 0, 1 \}^{\mathbb{T}^d_N}$ be a configuration space. An element $\eta = (\eta_1, \eta_2)  \in \mathcal{X}_N $ with $\eta_i = \{ \eta_i (x) \}_{ x \in \mathbb{T}^d_N }$ for $i =1,2$ denotes configuration of our two-species particle system with exclusion rules for type-$2$ particles. Namely, $\eta_1 (x) = k \in \mathbb{Z}_+ $ means there are $k$ type-$1$ particles on a site $x \in \mathbb{T}^d_N$, and $\eta_2 (x) = 1$ means there exists a type-$2$ particle on the site. For each $i = 1,2$, $\eta_i (x) = 0 $ means there is no type-$i$ particle on the site $x $. First we define an operator which governs the diffusion dynamics. For type-$1$ particles, we consider a zero-range operator $L_Z$ with jump rate $g :\mathbb{Z}_+ \to \mathbb{R}_+$ satisfying $g (0 ) = 0 $ and $g ( k ) > 0 $ for $k \ge 1 $, and an operator $L_K$ which corresponds to the Kawasaki process for type-$2$ particles, which are defined by 
\[
\begin{aligned}
& L_Z f (\eta_1, \eta_2) 
= \sum_{x, y \in \mathbb{T}^d_N  ,  |x-y| = 1 } g (\eta_1 (x)) [ f (\eta_1^{x,y}, \eta_2) - f(\eta_1, \eta_2)  ]  , \\
& L_K f (\eta_1, \eta_2 ) 
= \sum_{x, y \in \mathbb{T}^d_N ,  | x- y | = 1  } \eta_2(x ) ( 1 - \eta_2 (y ) )  [ f (\eta_1, \eta_2 ^{x,y } ) -  f ( \eta_1 , \eta_2 ) ]  
\end{aligned}
\] 
for every $f : \mathcal{X}_N \to \mathbb{R}$. Here for each $\sigma \in \mathcal{X}^1_{  N } $ or $\mathcal{X}^2_{ N } $, $\sigma^{ x , y } $ denotes the configuration after a particle on a site $x$ moves to a site $y$, which is defined by  
\[
\sigma^{x,y} (z)
= 
\begin{cases}
\begin{aligned}
& \sigma (x ) -1  && \text{ if } z = x, \\
& \sigma (y ) + 1 && \text{ if } z = y, \\
& \sigma (z) && \text{ otherwise. } \\
\end{aligned}
\end{cases}
\]
Note that looking the form of the Kawasaki generator neither jump from an empty site nor jump to an occupied site occur. Throughout this paper, we impose the following Lipschitz and linear growth conditions on the jump rate of zero-range process $g : \mathbb{Z}_+ \to \mathbb{R}_+ $. 

\begin{enumerate}
\item[\textbf{(LIP)}]
$\sup_{ k \ge 0  } | g ( k + 1 ) - g (k ) | < \infty $. 

\item[\textbf{(LG)}]
There exist constants $C> 0 $, $r_1 > 0 $ and $r_2 \in \mathbb{N} $ such that $g (k) \le C k $ and $g (k + r_2 ) - g( k ) \ge r_1 $ for every $k \in \mathbb{N}$. 
\end{enumerate}

Next we define a family of reference measures which are parametrized by density of each kind of particles. First, for given constant $\alpha \ge 0$, let $\overline{\nu}_{ 1, \alpha }$ be a probability measure on $ \mathbb{Z}_+ $ given by 
\[
\overline{ \nu }_{ 1, \alpha } ( k ) = \frac{1}{ Z_{ \alpha } } \frac{ \alpha^k }{ g (k) ! }, \quad 
Z_{ \alpha }  = \sum_{ k = 0 }^\infty \frac{ \alpha^k }{ g (k) ! } 
\]
where $g (k) ! = g (1) \cdots g (k)$ for $k \ge 1$ and $g (0) ! = 1$. Note that under the condition \textbf{(LG)}, the radius of convergence of $Z_\alpha $ becomes infinity so that the measure $\overline{\nu}_{ 1,  \alpha } $ is well-defined for every $\alpha \ge 0 $. Moreover, through a similar manner in \cite{KL99}, there exists a unique function $\varphi : \mathbb{R}_+ \to \mathbb{R}_+ $ such that $\rho = E_{ \overline{\nu}_{ 1, \varphi ( \rho )} } [\eta_1 (x) ] $ for each $\rho \ge 0 $. This is possible under our assumption \textbf{(LG)} and this determines a possibly non-linear smooth function $\varphi $ associated with the jump rate $g$. Hereafter we simply write $\nu_{1, \rho } \coloneqq \overline{ \nu }_{ 1, \varphi (\rho)} $. Note that an easy computation yields $\varphi^\prime (\rho ) = \varphi (\rho ) / \mathrm{Var}_{ \nu_\rho } [\eta_1 (x) ] $ so that the function $\varphi $ is increasing on $\mathbb{R}_+ $. Moreover, noting an expression $\rho ( \varphi ) = \varphi (d/ d \varphi ) \log Z_\varphi $, we have $\varphi^\prime (0 ) = 1 / ( d \rho / d \varphi ) (0) \ge c_0  $ for some constant $c_0 > 0$. We extend $\nu_{1, u } $ as a product measure $\nu_{1, u }  = \prod_{ x \in \mathbb{T}^d_N } \nu_{1, u (x) }$ for every $u = \{ u(x) \}_{ x \in \mathbb{T}^d_N } $ such that $u (x) \ge 0$ for each $x \in \mathbb{T}^d_N$. On the other hand, for any $v = \{ v (x) \} \in [0, 1]^{ \mathbb{T}^d_N }$, let $\nu_{2, v} $ be a product Bernoulli measure on $\mathcal{X}^2_N = \{ 0, 1 \}^{ \mathbb{T}^d_N } $ whose marginals are given by $\nu_{2 , v } (\eta_2 (x) = 1) = v(x) $ for every $x \in \mathbb{T}^d_N $. Finally, we define a spatially inhomogeneous product measure $\nu_{(u,v)}$ on $\mathcal{X}_N = \mathcal{X}_N^1 \times \mathcal{X}_N^2 $ by $\nu_{(u, v)} \coloneqq \nu_{1, u } \otimes \nu_{2, v } $ for each $u = \{ u(x) \}_{x \in \mathbb{T}^d_N}$, $v = \{ v (x) \}_{x \in \mathbb{T}^d_N}$ with $u(x) \ge 0$ and $0 \le v (x) \le 1 $ for every $x \in \mathbb{T}^d_N$. Throughout this paper, we fix a jump rate $g$ of the zero-range process satisfying the assumption \textbf{(LG)}. 

Next we define the Glauber operator $L_G $ by 
\[
\begin{aligned}
 L_G f (\eta_1, \eta_2 ) = 
 \sum_{ x \in \mathbb{T}^d_N } \eta_1 (x) \eta_2 (x) [ f (\eta_1^{x,-}, \eta_2^{x , - } ) - f (\eta_1, \eta_2) ]  
\end{aligned}
\]
where for each $\sigma \in \mathbb{Z}_+^{ \mathbb{T}^d_N } $ or $\{ 0, 1 \}^{ \mathbb{T}^d_N } $ the transition $\sigma \mapsto \sigma^{x,-}$ decreases the number of corresponding type of particles on a site $x \in \mathbb{T}^d_N$ if possible, which is defined by 
\[
\sigma^{x,-} (z)
= 
\begin{cases}
\begin{aligned}
& \sigma (x) - 1  && \text{ if } z = x, \\
& \sigma (z) && \text{ otherwise. } \\
\end{aligned}
\end{cases}
\]
Here we also note that the above killing action occurs only when there is at least each of one particle for both of two types on a site $x$ since otherwise we have $\eta_1 (x) \eta_2 (x) = 0$. Let $K = K(N)$ and $\varepsilon = \varepsilon (N) $ be parameters which satisfy the following bounds as the scaling parameter $N$ tends to infinity.

\begin{enumerate}[(B1)]
\item \label{K}
$K = K (N) $ diverges as $N$ tends to infinity satisfying $1 \le K(N) \le ( \delta_1 \log ( \delta_2 \log N) )^{1/2}$ with $\delta_1 , \delta_2 > 0$ small enough. 

\item \label{Epsilon}
$\varepsilon = \varepsilon (N) $ vanishes as $N$ tends to infinity satisfying $\varepsilon (N) \ge N^{ - \alpha_\varepsilon  } $ for some small constant $\alpha_\varepsilon > 0 $. 



\end{enumerate}
Then, let $\{ \eta^N_t = (\eta^N_{1,t } , \eta^N_{2, t } ) \}_{t \ge 0}$ be a $\mathcal{X}_N $-valued Markov process with generator $L_N = N^2 L_Z + \varepsilon (N) L_K + K (N) L_G$ and initial distribution $\mu^N_0 $ on $\mathcal{X}_N $ which is constructed on some probability space $(\Omega^N ,\mathcal{F}^N, \mathbb{P}^N_{ \mu^N_0 })$. When the initial distribution $\mu^N_0 $ is clear from the context, we omit the dependence on it. We denote the expectation with respect to $\mathbb{P}^N_{\mu^N_0 } $ (or $\mathbb{P}^N $) by $\mathbb{E}^N_{ \mu^N_0 } $ (or $\mathbb{E}^N $).

\begin{remark}
A reason why we introduced the weak diffusion for type-$2$ particles with intensity $\varepsilon (N) $ is mainly for using spectral gap estimates to prove the Boltzmann-Gibbs principle, which will be explained later (see Section \ref{sec:BG}). The slower the parameter $\varepsilon (N) $ decays, the easier the proof of the Boltzmann-Gibbs principle becomes. We can find that the above polynomial decay is enough for that. The bound for the reaction rate $K(N) $ is also needed for our strategy of proof. 
\end{remark}

Now we state some assumptions on the initial distribution $\mu^N_0 $. For that purpose, we prepare some notations. For two probability measures $\nu$ and $\mu$ on $\mathcal{X}_N$, define relative entropy of $\mu$ with respect to $\nu$ by
\[
H (\mu | \nu) = \int_{ \mathcal{X}_N } \frac{ d\mu }{ d\nu } \log \frac{ d\mu }{d\nu} d\nu 
\]
if $\mu$ is absolutely continuous with respect to $\nu$, and otherwise we let $H (\mu | \nu) = \infty$. Moreover, we define the discrete gradient $\nabla^N$ by $\nabla^N = {}^{t} \! ( \partial^N_1 , \ldots , \partial^N_d )$ with $\partial^N_j u (x )  = N (u ( x + e_j ) - u (x) ) $ for any $j = 1, \ldots ,d $ and any $u = \{ u (x) \}_{ x \in \mathbb{T}^d_N} $. Then we assume that there exist functions $u^N_0, v^N_0 $ on $\mathbb{T}^d_N$ and functions $u_0 , v_0 $ on $ \mathbb{T}^d  $ such that the following conditions hold.


\begin{enumerate}[({A}1)]
\item \label{IF}
There exist constants $C_0, C_1 > 0$, $M_u \ge 0$ and $0 \le M_v < 1 $ such that 
\[
\begin{aligned}
& e^{- C_1 K (N) } \le u^N_0 (x) \le M_u ,\quad  | \nabla^N u^N_0 ( x ) | \le C_0 K (N) , \quad |  \partial^N_i \partial^N_j u^N_0 (x)  | \le C_0 K (N)^3  \\
& e^{- C_1 K (N) } \le v^N_0 (x) \le M_v , \quad | \nabla^N v^N_0 ( x ) | \le C_0 \varepsilon^{-1/2 } K (N ) 
\end{aligned}
\]
for every $x \in \mathbb{T}^d_N$ and $i, j = 1 , \ldots, d $.  

\item \label{RE}
For the product measure $\nu^N_0 = \nu_{ ( u^N_0, v^N_0 ) } \in \mathcal{P} (\mathcal{X}_N ) $ with weight  $( u^N_0, v^N_0) $ in the assumption (A\ref{IF}), the initial distribution $\mu^N_0 \in \mathcal{P} (\mathcal{X}_N ) $ satisfies $H (\mu^N_0 | \nu^N_0 ) = O (N^{ d- \delta_0 } ) $ for some small positive constant $\delta_0 $ as $N$ tends to infinity, that is, there exists a positive constant $C$ such that $ H (\mu^N_0 | \nu^N_0) \le C N^{d - \delta_0 } $ for every sufficiently large $N$. 

\item \label{PDE}
$ u_0 , v_0 \in L^2 (\mathbb{T}^d ) $ and $u^N_0 \rightharpoonup u_0 $ and $v^N_0 \rightharpoonup v_0 $ weakly in $L^2 (\mathbb{T}^d) $ as $N$ tends to infinity. 
\end{enumerate}


\begin{remark}
As discussed in \cite{H20} there exists non-trivial macroscopic initial functions $u_0 $ and $v_0$ and microscopic approximating sequences $\{ u^N_0 \}_{N \in \mathbb{N} } $ and $\{ v^N_0 \}_{ N \in \mathbb{N} } $ satisfying the above assumptions (A\ref{IF}), (A\ref{RE}) and (A\ref{PDE}) (though \cite{H20} imposed a restriction only for the first derivative growth, we can construct an example in a similar way). 
\end{remark}

\subsection{Main result}
We define empirical measures associated with the Markov process $\{ \eta^N_t : t \in [0, T ] \} $ by 
\[
\pi^N_{ i, t } (d \theta )  = N^{ - d } \sum_{ x \in \mathbb{T}^d_N } \eta^N_{ i, t } (x) \delta_{ x / N } (d \theta ) , \quad 
i = 1, 2 . 
\]
Here $\delta $ stands for the Dirac delta measure on $\mathbb{T}^d $. We write $\langle \pi^N_t , \psi \rangle = ( \langle \pi^N_{ 1, t } , \psi_1 \rangle, \langle \pi^N_{ 2 , t } , \psi_2 \rangle ) $ with $\langle \pi^N_{i, t } , \psi_i \rangle = \int_{ \mathbb{T}^d } \psi_i (\theta ) \pi^N_{i,  t } (d \theta ) $ for every $\psi = (\psi_1 ,\psi_2 ) \in C (\mathbb{T}^d  , \mathbb{R}^2  ) $. To state the main theorem in this paper, we introduce the notion of weak solution of the following Stefan problem. 
\begin{equation}
\begin{cases}
\begin{aligned}
\label{wStefan}
& \partial_t w = \Delta \mathcal{D}_\varphi (w)  &&\text{ in } Q_T , \\
& w ( 0, \cdot ) = u_0 - v_0  && \text{ on } \mathbb{T}^d 
\end{aligned}
\end{cases}
\end{equation}
where $\mathcal{D}_\varphi $ is a function on $\mathbb{R} $ defined by $\mathcal{D}_\varphi (s) = \varphi (s) \mathbf{1}_{ [0, \infty ) } (s) $ and $Q_T \coloneqq [ 0, T ] \times \mathbb{T}^d $.

\begin{definition}
\label{def Stefan}
A function $w \in L^\infty (Q_T)$ is called a weak solution of the problem \eqref{wStefan} with initial function $w_0 \in L^\infty (\mathbb{T}^d ) $ if the following conditions (i) and (ii) hold. \\
(i) $\varphi (w_+) \in L^2 (0,T; H^1 (\mathbb{T}^d ) )$, \\
(ii) For every $T> 0$, the function $ w $ satisfies an integral identity 
\begin{equation}
\label{weakform}
\int_{ \mathbb{T}^d } w_0 (\theta ) \psi (0, \theta) d\theta + \int_0^T \int_{ \mathbb{T}^d } \big( w_+ \partial_t \psi - \nabla \varphi (w_+ ) \cdot \nabla \psi \big) (t, \theta ) d\theta dt  = 0
\end{equation}
for any test function $\psi \in L^2 (0, T ; H^1 (\mathbb{T}^d ) ) $ such that $\psi (T, \cdot) \equiv 0 $. 
\end{definition}

Let $w$ be a weak solution of the problem \eqref{wStefan} and suppose functions $ u \coloneqq w_+$ and $v \coloneqq w_- $ satisfy $\text{supp}(u ( t,\cdot ) ) \cap \text{supp} (v(t, \cdot ) ) = \phi$, and the closure of the region $ \text{supp}(u ( t,\cdot ) ) \cup \text{supp} (v(t, \cdot ) ) $ is $\mathbb{T}^d $ for every $t \in [ 0, T ] $. Here $ w_+ = \max \{w, 0 \}  $ and $w_- = \max \{ - w , 0 \} $ are positive and negative part of $w$, respectively. Moreover we define
\[
\begin{aligned}
 \Omega^u \coloneqq \bigcup_{ t \in (0, T) } \{ t \} \times \Omega^u (t) , \quad 
\Omega^v \coloneqq \bigcup_{ t \in (0, T) } \{ t \} \times \Omega^v (t) , \quad
 \Gamma \coloneqq \bigcup_{t \in (0, T) } \{ t \} \times \Gamma (t)
\end{aligned}
\]
for $\Omega^u (t) \coloneqq \{ u (t, \cdot ) > 0 \}$,  $\Omega^v (t) \coloneqq  \{ v (t, \cdot ) > 0 \}$ and $\Gamma (t) \coloneqq \partial \Omega^u (t) = \partial \Omega^v (t)$. If $u $ is smooth in $\Omega^u (t) $ and $\Gamma (t )$ is smooth, then we have the following strong form of one-phase Stefan problem.  
\begin{equation*}
\begin{cases}
\begin{aligned}
& \partial_t u = \Delta \varphi (u) && \text{ in } \Omega^u,  \\
& v = v_0 && \text{ in } \Omega^v,  \\
& u = 0 && \text{ on } \Gamma , \\ 
& u (0, \cdot ) = u_0 (\cdot ) , \quad v (0, \cdot ) = v_0 (\cdot)     && \text{ on } \mathbb{T}^d .
\end{aligned}
\end{cases}
\end{equation*}
Using the above notion of weak solutions, our main theorem can be stated as follows.

\begin{theorem}
\label{main thm}
Assume initial functions $u^N_0 $, $v^N_0 $, $u_0 $ and $v_0$ satisfy (A\ref{IF}), (A\ref{RE}) and (A\ref{PDE}), and assume the scaling parameters $K = K (N) $ and $\varepsilon = \varepsilon (N) $ satisfy (B\ref{K}) and (B\ref{Epsilon}).  Then there exist non-negative functions $u $ and $v$ on $ Q_T $ such that $u (t, \cdot ) v (t, \cdot ) = 0 $ a.e. on $\mathbb{T}^d $ for each $t \in ( 0, T ] $, and for every $\psi \in C^\infty (Q_T; \mathbb{R}^2 )$ and $\varepsilon > 0 $ we have 
\[
\lim_{N \to \infty } \mathbb{P}^N_{ \mu^N_0 } \bigg( \bigg| \int_0^T \big( \langle \pi^N_t , \psi (t, \cdot ) \rangle  - \langle (u(t, \cdot) , v (t, \cdot) ) , \psi (t , \cdot ) \rangle_{ L^2 (\mathbb{T}^d ; \mathbb{R}^2 ) } \big) dt \bigg| > \varepsilon  \bigg) = 0 .
\]
Moreover, $w \coloneqq u - v $ is a unique weak solution of the Stefan problem \eqref{wStefan}. In particular, the spatially segregated functions $u $ and $v$ are uniquely determined from the function $w$. 
\end{theorem}


\section{Strategy of proof}
\label{sec:main thm}
\subsection{Yau's relative entropy method}
In this section, we give the proof of the main theorem (Theorem \ref{main thm}) based on Yau's relative entropy method which is introduced in \cite{Yau91}. We first concern a probability part which asserts that at any give time, the microscopic density profile is close to the macroscopic one in the sense of relative entropy (see Theorem \ref{prob thm} below). Let $u^N = \{ u^N (t,x) \}_{ t \in [0, T] , x \in \mathbb{T}^d_N}$, $v^N = \{ v^N (t,x) \}_{ t \in  [0,T] , x \in \mathbb{T}^d_N}$ be a pair of solutions of the following semi-discretized reaction-diffusion system. 
\begin{equation}
\label{dHDL eq}
\begin{cases}
\begin{aligned}
&\partial_t u^N (t,x) = \Delta^N \varphi( u^N (t,x) ) - K (N) u^N (t,x) v^N (t,x)  && \text{ in } (0, T) \times \mathbb{T}^d_N , \\
&\partial_t v^N (t,x) = \varepsilon (N) \Delta^N u^N (t, x) - K (N) u^N (t,x) v^N (t,x) && \text{ in } (0, T) \times \mathbb{T}^d_N , \\
& u^N (0, x) = u^N_0 (x) , \quad v^N (0, x) = v^N_0 (x) && \text{ on }  \mathbb{T}^d_N 
\end{aligned}
\end{cases}
\end{equation}
with initial functions $ u^N_0 $ and $ v^N_0 $ on $ \mathbb{T}^d_N $ satisfying the assumptions (A\ref{IF}), (A\ref{RE}) and (A\ref{PDE}). Here, $\Delta^N$ is the discrete Laplacian defined by 
\[
\Delta^N u (x) = N^2 \sum_{ 1 \le j \le  d } \big[ u (x + e_j ) + u (x - e_j) - 2 u (x)  \big] 
\]
for any $u = \{ u (x) \}_{ x \in \mathbb{T}^d_N } $ where $\{ e_j \}_{ j = 1,\ldots ,d }$ is the normal basis on $\mathbb{Z}^d$. Let $\nu^N_t = \nu_{ (u^N (t) , v^N (t) ) }$ be the spatially inhomogeneous product measure on the configuration space $\mathcal{X}_N$ with $ u^N (t) = \{ u^N (t, x) \}_{ x \in \mathbb{T}^d_N } $ and $ v^N (t) = \{ v^N (t, x) \}_{ x \in \mathbb{T}^d_N } $, and let $\mu^N_t $ be the probability law of $\eta^N_t = (  \eta^N_{1,t}, \eta^N_{2, t} ) $ on $\mathcal{X}_N$. According to a comparison argument we give in Section \ref{sec:PDE estimates} (see Lemma \ref{unif est}), we can show that the semi-discretized solutions $u^N $ and $v^N$ take values in $\mathbb{R}_+$ and $[0, 1]$, respectively, which makes the product measure $\nu^N_t $ well-defined. Then we can show the following key estimate which states the microscopic measure $\mu^N_t $ is asymptotically close to the measure $\nu^N_t $. 

\begin{theorem}
\label{prob thm}
Assume (A\ref{IF}), (A\ref{RE}), (B\ref{K}) and (B\ref{Epsilon}). Then for every $t \ge 0$ we have $ H (\mu^N_t | \nu^N_t) = o (N^d) $ as $N$ tends to infinity. 
\end{theorem}

Next we extend $u^N$ and $v^N$ on $ Q_T  = [0,T]\times \mathbb{T}^d$ as simple functions, the extensions are denoted again with the same notation, by
\begin{equation}
\label{extension}
\begin{aligned}
& u^N (t, \theta ) = \sum_{ x \in \mathbb{T}^d_N } u^N (t, x) \prod_{ 1 \le j \le d }\mathbf{1}_{ \left[ \frac{x_j}{ N }  - \frac{1}{ 2N } , \frac{ x_j}{ N} + \frac{ 1}{ 2N} \right) } (\theta_j) , \\ 
& v^N (t, \theta ) = \sum_{ x \in \mathbb{T}^d_N } v^N (t, x) \prod_{ 1 \le j \le d } \mathbf{1}_{ \left[ \frac{x_j}{ N }  - \frac{1}{ 2N } , \frac{ x_j}{ N} + \frac{ 1}{ 2N} \right) } (\theta_j)  \\ 
\end{aligned}
\end{equation}
for every $t \in [0,T]$ and $\theta = (\theta_j)_{ j = 1, \ldots , d} \in \mathbb{T}^d$. For these extended functions $u^N$ and $v^N $ on $Q_T $, we can show the following convergence result, which is purely a deterministic problem. 


\begin{theorem}
\label{PDE thm}
Assume (A\ref{IF}), (A\ref{PDE}), (B\ref{K}) and (B\ref{Epsilon}). Let $u^N $ and $v^N$ be functions on $Q_T $ defined by \eqref{extension}. There exist functions $u$ and $v$ on $Q_T $ such that $u(0, \cdot) = u_0 (\cdot )  $, $v ( 0, \cdot ) = v_0 (\cdot ) $ and 
\[
\begin{aligned}
& u \in L^{\infty } (Q_T ) \cap L^2 (0, T; H^1 (\mathbb{T}^d ) ) , \, v \in L^\infty (Q_T ) , \\
& 0 \le u \le M_u,  \,   0 \le v \le M_v \text{ and } uv = 0 \text{ a.e. in } Q_T,  \\
& u^N \to u \text{ strongly in } L^2 (Q_T) \text{ and a.e. in } Q_T, \\
& v^N \rightharpoonup v \text{ weakly in } L^2 (Q_T)  
\end{aligned}
\]
as $N$ tends to infinity. Moreover, $w \coloneqq u - v $ is a unique weak solution of the problem \eqref{wStefan} with initial function $u_0 - v_0 $. 
\end{theorem}

Theorem \ref{prob thm} and Theorem \ref{PDE thm} will be proved in Section \ref{sec:prob} and Section \ref{sec:PDE}, respectively.

\subsection{Proof of Theorem \ref{main thm}} 
Once Theorem \ref{prob thm} and Theorem \ref{PDE thm} are proved, we can show the main theorem as follows. Let $u^N $ and $v^N $ be the functions on $Q_T = [0, T] \times \mathbb{T}^d $ defined by (\ref{extension}). For any $\varepsilon > 0$ and any smooth function $\psi \in C^\infty ( \mathbb{T}^d )$, let us define 
\[
\begin{aligned}
& \mathcal{A}_1 = \mathcal{A}_1 ( \psi, \varepsilon  ) \coloneqq \{ \eta \in \mathcal{X}_N ; \left| \langle \pi^N_{1, t} , \psi \rangle - \langle u^N (t,\cdot) , \psi \rangle_{L^2(\mathbb{T}^d) }  \right| > \varepsilon \},  \\
& \mathcal{A}_2 = \mathcal{A}_2 ( \psi, \varepsilon  ) \coloneqq \{ \eta \in \mathcal{X}_N ; \left| \langle \pi^N_{2, t} , \psi \rangle - \langle v^N (t,\cdot) , \psi \rangle_{L^2(\mathbb{T}^d) }  \right| > \varepsilon \} .
\end{aligned}
\]
Then, as a corollary of the entropy inequality, we get
\[
\mu^N_t (\mathcal{A}_i) \leq \frac{\log 2 + H ( \mu^N_t | \nu^N_t) }{\log(1+ 1/ \nu^N_t (\mathcal{A}_i))} , 
\quad i =  1, 2 . 
\]
The next result is a direct consequence of the exponential Markov inequality. 

\begin{lemma}
\label{nu est}
For any $\psi \in C^\infty (\mathbb{T}^d )$ and $\varepsilon > 0$, there exists a positive constant $C= C (\varepsilon, \| \psi \|_{  L^\infty (\mathbb{T}^d ) })$ such that
\[
\nu^N_t (\mathcal{A}_i (\psi,  \varepsilon ) ) \le \exp ( -C N^d ) .  
\]
In particular, the above estimate holds uniformly in $\{ \psi ;  \| \psi \|_{ L^\infty (\mathbb{T}^d ) } < M \} $ for every $M >0$. 
\end{lemma}

Then by Theorem \ref{prob thm}, we have 
$
\lim_{N \to \infty} \mu^N_t (\mathcal{A}_i ( \psi_i (t, \cdot), \varepsilon  ) ) = 0
$
for each $i=1,2$, $t \in [0,T]$, $\varepsilon > 0$ and $\psi = (\psi_1, \psi_2 )  \in C^\infty ( Q_T ; \mathbb{R}^2 )$.

Now we estimate the probability appearing in the main theorem (Theorem \ref{main thm}) by using Markov's inequality and the triangle inequality by 
\begin{equation}
\label{main thm proof}
\begin{aligned}
& \varepsilon^{ - 1 } \int_0^T E_{ \mu^N_t } \big[  | \langle \pi^N_t , \psi (t, \cdot )  \rangle - \langle ( u^N  (t, \cdot ), v^N  (t, \cdot ) ) , \psi \rangle_{ L^2 (\mathbb{T}^d ) } | \big] dt \\
& + \varepsilon^{-1 } \bigg| \int_0^T  \langle (u^N (t, \cdot ) -u (t, \cdot ), v^N (t, \cdot ) - v (t, \cdot )) , \psi  (t, \cdot ) \rangle_{ L^2 (\mathbb{T}^d) } dt \bigg| .   
\end{aligned}
\end{equation}
We can see that these two terms vanish as $N $ tends to infinity. Indeed, since $u^N \rightharpoonup u $ and $v^N \rightharpoonup v$ weakly in $L^2 (Q_T)$ by Theorem \ref{PDE thm}, the second term in \eqref{main thm proof} vanishes as $N$ tends to infinity. On the other hand, the integrand in the first term can be bounded above by 
\begin{equation}
\label{mainthmproof2}
\begin{aligned}
E_{ \mu^N_t } \bigg[  \big| \langle \pi^N_t , \psi \rangle - \langle ( u^N (t, \cdot ), v^N (t, \cdot ) ) , \psi (t, \cdot) \rangle_{ L^2 (\mathbb{T}^d ) } \big| , \bigcap_{ i = 1, 2 } \mathcal{A}_i (\psi_i ( t, \cdot ), \tilde{ \varepsilon } ) \bigg] + \tilde{\varepsilon} 
\end{aligned}
\end{equation}
for any $\tilde{ \varepsilon } > 0 $. As we will see in forthcoming sections, we have 
\[
\sup_{ N \in \mathbb{N} } E_{ \mu^N_t } [| \langle \pi^N_t , \psi \rangle |] 
\le \sup_{ N \in \mathbb{N} } \| \psi \|_{\infty} N^{ -d } E_{ \mu^N_t } \big[ \sum_{ x \in \mathbb{T}^d_N } \eta^N_t (x) \big] \le C \| \psi \|_\infty  
\]
by Lemma \ref{sum est}, and $ \sup_{ N \in \mathbb{N} } | \langle (u^N, v^N ) , \psi \rangle | \le \| \psi \|_{\infty } (M_u \vee M_v ) $ by Lemma \ref{unif est}. Since $\lim_{N \to \infty } \mu^N_t (\mathcal{A}_i ) = 0 $ for each $i = 1, 2$ as we proved in the above, the first term in \eqref{mainthmproof2} vanishes as $N $ goes to infinity. Hence we finish the proof since $\tilde{ \varepsilon } $ was arbitrary.

\section{Proof of Theorem \ref{prob thm}}
\label{sec:prob}
In this section, we give the proof of Theorem \ref{prob thm} as a probabilistic problem. For that purpose, we estimate the derivative of the relative entropy $H (\mu^N_t | \nu^N_t ) $ in time argument which is called the entropy production. First we define a Dirichlet energy for our diffusion dynamics. For any non-negative function $f $ and probability measure $\nu$ on the configuration $\mathcal{X}_N$, define
\[
\begin{aligned}
& \mathcal{D}_Z (f ; \nu) \coloneqq \frac{ 1 }{ 2 } \sum_{ x \in \mathbb{T}^d_N} \sum_{ 1 \le j \le d } \int_{ \mathcal{X}_N }  
g (\eta_1 (x) ) \big[ f (\eta_1^{x, x+e_j}, \eta_2) - f ( \eta_1, \eta_2 ) \big]^2 d\nu (\eta_1, \eta_2) , \\
& \mathcal{D}_K (f ; \nu) \coloneqq \frac{1}{2} \sum_{ x \in \mathbb{T}^d_N} \sum_{ 1 \le j \le d } \int_{ \mathcal{X}_N }  
\eta_2 (x) (1 - \eta_2 (y) ) \big[ f (\eta_1, \eta_2^{x, x+e_j} ) - f ( \eta_1, \eta_2 ) \big]^2 d\nu (\eta_1, \eta_2) \\
\end{aligned}
\]
where $\{ e_j \}_{ j = 1,\ldots ,d } $ is the normal basis of $\mathbb{Z}^d$. Then define $\mathcal{D} = 2 N^2 \mathcal{D}_Z + 2 \varepsilon (N ) N^2 \mathcal{D}_K $. We have the following estimate of the entropy production.

\begin{proposition}[Yau's inequality, \cite{JM18}]
\label{Yau}
For any probability measures $\{ \nu_t \}_{t \ge 0} $ and $m$ on $\mathcal{X}_N $ which are differentiable in $t$ and full-supported, we have
\[
\frac{d}{dt} H (\mu^N_t | \nu_t) \le -2 N^2 \mathcal{D} \bigg( \sqrt{ \frac{ d\mu^N_t }{ d\nu_t  } } ; \nu_t \bigg)  
+ \int_{ \mathcal{X}_N } ( L^{*, \nu_t}_N \mathbf{1} - \partial_t \log \psi_t )d\mu^N_t 
\]
where $L^{*, \nu_t}_N$ is the adjoint operator of $L_N$ on $L^2 (\nu_t)$ and $\psi_t \coloneqq d\nu_t / dm$. 
\end{proposition}

The main ingredients of this section is given as follows. 

\begin{theorem}
\label{integrand est}
Assume the same conditions as in Theorem \ref{prob thm}. Define $f^N_t \coloneqq d \mu^N_t / d \nu^N_t  $ and $\psi^N_t \coloneqq d\nu^N_t / dm $ where $m$ is an arbitrary full-supported probability measure on $\mathcal{X}_N $. Then there exist constants $C > 0$ and small $\varepsilon_0 > 0$ such that 
\[
\begin{aligned}
\int_0^T \int_{ \mathcal{X}_N } ( L^{*, \nu_t}_N \mathbf{1} - \partial_t \log \psi^N_t )d\mu^N_t dt 
\le 
& \int_0^T  \varepsilon N^2 \mathcal{D}_K \big(\sqrt{ f^N_t  } ; \nu^N_t \big) dt \\
& + C K^3 e^{C K^2 } \int_0^T H (\mu^N_t | \nu^N_t ) dt 
 + O ( K^4 e^{C K^2 } N^{d - \varepsilon_0 }  )
\end{aligned}
\]
as $N$ tends to infinity.  
\end{theorem}

Now combining the above main estimate and Yau's inequality (Proposition \ref{Yau}) with $\nu_t = \nu^N_t$ and $\psi_t = \psi^N_t = d \nu^N_t / dm $, we can give the proof of the probabilistic part (Theorem \ref{prob thm}). 

\begin{proof}[Proof of Theorem \ref{prob thm}]
We apply Yau's inequality with $\nu = \nu^N_t $ integrated over $[0, t] $. Then, the time integral of the Dirichlet energy corresponding to Kawasaki dynamics appearing in the right-hand side of the estimate given in Theorem \ref{integrand est} can be absorbed into the total Dirichlet energy with negative sign in Yau's inequality. Therefore, there exists a positive constant $C $ such that  
\[
H (\mu^N_t | \nu^N_t ) \le  H (\mu^N_0 | \nu^N_t ) + C K^3 e^{C K^2 } \int_0^t H (\mu^N_s | \nu^N_s ) ds + O (K^4 e^{C K^2} N^{d -\varepsilon_0 } )  
\]
as $N$ tends to infinity. By applying Gronwall's inequality, noting $H (\mu^N_0 | \nu^N_0 ) = O (N^{d- \delta_0 } ) $ by the assumption (A\ref{RE}), there exists a positive constant $C = C (T ) $ such that   
\[
H (\mu^N_t | \nu^N_t ) \le \big( H (\mu^N_0 | \nu^N_0 ) + O (K^4 e^{CK^2 } N^{d-\varepsilon_0 } ) \big) e^{ C K^3 e^{C K^2 }  } 
\]
for every $t \in [0, T]$. Now we take $K = K (N) $ so that $K^2 \le \delta_1 \log (\delta_2 \log N ) $ with $\delta_1, \delta_2 > 0$ satisfying $C \delta_1 \le 1 $ and $\delta_2 \le (\delta_0 \wedge \varepsilon_0 )/ 2$ by assumption (B\ref{K}). Then we have $e^{C( T \vee 1 ) K^2 } \le \delta_2 \log N $ so that $H (\mu^N_t | \nu^N_t ) = O( N^{d - (\delta_0 \wedge \varepsilon_0 ) /2 } )$, which is the desired assertion in Theorem \ref{prob thm}. 
\end{proof}

Hence our task is now to show the main estimate Theorem \ref{integrand est}.

In this subsection, we see that one can linearize some terms in the integrand $L^{*, \nu^N_t }_N \mathbf{1} - \partial_t \log \psi^N_t $ with the help of so called the `multi-variable Boltzmann-Gibbs principle' and give the proof of Theorem \ref{integrand est}. To see that, fix an arbitrary local function $h$ on $\mathcal{X}_N $ with support in a finite square box $\Lambda_h \Subset \mathbb{T}^d_N$ and parameter $\beta = (\beta_1, \beta_2 )$ with $\beta_1 \ge 0$ and $0 \le \beta_2 \le 1$. We write  
\[
\tilde{h} (\beta) \coloneqq E_{\nu_\beta} [h] .
\] 
Assume that there exists a positive constant $C$ such that 
\begin{equation}
\label{BG assumption}
|h (\eta)| \le C \sum_{y \in \Lambda_h } \big( g_1 (y, \eta_2 (y) ) \eta_1 (y) + g_2 (y, \eta_2 (y) )  \big)   
\end{equation}
for every $\eta = ( \eta_1 , \eta_2 ) \in \mathcal{X}_N $ where $g_1 (y, \eta_2 ) $ and $g_2 (y, \eta_2 ) $ are polynomials of $\{ \eta_2 (z) ; z \in \Lambda_h \} $ for each $y \in \Lambda_h $. Given such a local function $h$, the Boltzmann-Gibbs principle enables us to extract the ``linear part'' of $h$. To see that, we define a residual function $R : \mathbb{T}^d_N \times \mathcal{X}^1_N \times \mathcal{X}^2_N \to \mathbb{R}$ of $h$ by 
\[
\begin{aligned}
R ( x,  \eta_1, \eta_2 ) \coloneqq 
& \tau_x h ( \eta_1 , \eta_2 ) - \tilde{h} ( u^N (t, x ), v^N (t, x ) )  \\
&  - \partial_1 \tilde{h} (u^N (t, x), v^N (t, x) ) ( \eta_1 (x) - u^N (t, x) ) \\
& - \partial_2 \tilde{h} (u^N (t, x), v^N (t, x) ) ( \eta_2 (x) - v^N (t, x) ) 
\end{aligned}
\]
where $u^N$ and $v^N$ are solutions to the semi-discretized reaction-diffusion system \eqref{dHDL eq} and $\partial_i$ denotes the derivative with respect to the $i$-th component ($i=1,2$) for every differentiable function on $\mathbb{R}^2$. In the sequel, we  omit the dependence on configuration and simply write as $R (x) = R (x , \eta_1, \eta_2 )$. Then the multi-variable Boltzmann-Gibbs principle is stated as follows.

\begin{theorem}[The multi-variable Boltzmann-Gibbs principle]
\label{BG}
Assume (A\ref{IF}), (B\ref{K}) and (B\ref{Epsilon}). Moreover, assume $H (\mu^N_0 | \nu^N_0 ) = O (N^d)$. Then there exist positive constants $C$ and $\varepsilon_0$ such that
\[
\mathbb{E}^N_{ \mu^N_0 } \bigg[ \bigg| \int_0^T  \sum_{x \in \mathbb{T}^d_N } R ( x )  dt  \bigg|  \bigg] 
\le \int_0^T C  H(\mu^N_t | \nu^N_t )dt +  O (K N^{d-\varepsilon_0} ) 
\]
as $N$ tends to infinity. Moreover, we may take $\varepsilon_0 = d / ( 3 d + 1 )$. 
\end{theorem}

We postpone the proof of Theorem \ref{BG} in Section \ref{sec:BG} and give a proof of Theorem \ref{integrand est} here. 

\begin{proof}[Proof of Theorem \ref{integrand est}]
First we introduce rescaled variables 
\[
\omega_{1, x} \coloneqq \frac{\eta_1 (x) - u^N (t, x)  }{ \chi_1 ( u^N (t, x) ) }, \quad
\omega_{2, x} \coloneqq \frac{\eta_2 (x) - v^N (t, x)  }{ \chi_2 (v^N (t, x ) ) }  
\]
where $\chi_1 (\rho_1) = \text{Var}_{\nu_{ 1, \rho_1 } } [ \eta_1 (0) ] = \varphi (\rho_1) / \varphi^\prime ( \rho_1) $ and $\chi_2 (\rho_2) =\text{Var}_{\nu_{ 2, \rho_2  } } [ \eta_2 (0) ] = \rho_2 (1- \rho_2 ) $ for every $\rho_1 \ge 0$ and $0 \le \rho_2 \le 1$. The Boltzmann-Gibbs principle enables us to extract leading terms, and to rewrite the integrand $L^{*, \nu^N_t }_N \mathbf{1} - \partial \log \psi^N_t $ into a linear combination of $\omega$, asymptotically as $N $ tends to infinity. As one can easily verify, the integrand can be calculated as 
\begin{equation}
\label{integrand terms}
\begin{aligned}
& L^{*, \nu^N_t}_N \mathbf{1} - \partial_t \log \psi^N_t \\
&\quad  = \sum_{x \in \mathbb{T}^d_N } \frac{ \Delta^N \varphi (u^N (t,x ) ) }{ \varphi (u^N (t,x) ) } \big( g (\eta_1 (x)) - \varphi (u^N (t,x) ) \big) \\
& \qquad - \frac{ \varepsilon N^2 }{ 2 } \sum_{x, y \in \mathbb{T}^d_N ; |x - y | = 1 } \big( v^N (t, x) - v^N (t, y) \big)^2 \omega_{2, x } \omega_{2, y }  + \varepsilon  \sum_{ x \in \mathbb{T}^d_N } \Delta^N v^N (t, x) \omega_{2, x } \\
& \qquad + K  \sum_{x \in \mathbb{T}^d_N } \bigg( (\eta_1 (x) + 1)  \frac{ \varphi (u^N (t, x) ) }{ g (\eta_1 (x) +1) } (1- \eta_2 (x) ) \frac{ v^N (t, x ) }{1 -v^N (t, x ) } - \eta_1 (x) \eta_2 (x)  \bigg) \\
& \qquad - \sum_{x \in \mathbb{T}^d_N } \partial_t u^N (t,x ) \omega_{1,x } - \sum_{x \in \mathbb{T}^d_N } \partial_t  v^N (t,x) \omega_{2, x }  .
\end{aligned}
\end{equation}
In this identity, we let possibly non-linear terms in $\omega$ as 
\[
\begin{aligned}
& I_1 \coloneqq \sum_{x \in \mathbb{T}^d_N } \frac{ \Delta^N \varphi (u^N (t,x ) ) }{ \varphi (u^N (t,x) ) } \big( g (\eta_1 (x)) - \varphi (u^N (t,x) ) \big) \\
& I_2 \coloneqq - \frac{ \varepsilon N^2 }{ 2 } \sum_{x, y \in \mathbb{T}^d_N ; |x - y | = 1 } \big( v^N (t, x) - v^N (t, y) \big)^2 \omega_{2, x } \omega_{2, y } \\
& I_3 \coloneqq K  \sum_{x \in \mathbb{T}^d_N } \bigg( (\eta_1 (x) + 1)  \frac{ \varphi (u^N (t, x) ) }{ g (\eta_1 (x) +1) } 
\frac{ v^N (t, x ) }{1 -v^N (t, x ) }  - \eta_1 (x)   \bigg) \eta_2 (x) .
\end{aligned}
\]
We apply the multi-variable Boltzmann-Gibbs principle (Theorem \ref{BG}) to replace them by linear combination of rescaled variables $\omega_1 $ and $\omega_2 $. First, for $I_1$ let $h(\eta ) = g (\eta_1 (x) ) - \varphi (u^N (t,x) )$. Then we have $\tilde{h} (\beta ) = \varphi (\beta) - \varphi (u^N (t, x ) )$ as $\tilde{g} (\beta) = E_{\nu_{ \beta } } [g] = \varphi (\beta) $, which implies $\tilde{h} (u^N (t, x), v^N (t, x ) ) = 0$, $\partial_1 \tilde{h} (u^N (t, x), v^N (t, x ) ) = \varphi^\prime (u^N (t, x) ) $ and $\partial_2 \tilde{h} (u^N (t, x ), v^N (t, x  ) ) = 0 $. Therefore, by the Boltzmann-Gibbs principle applied to this local function $h$, there exists a positive constant $C$ so that
\[
\begin{aligned}
& \mathbb{E}^N_{  \mu^N_0 } \bigg[  \bigg|  \int_0^T \sum_{x \in \mathbb{T}^d_N }  \{ g (\eta_1 (x) ) - \varphi (u^N (t, x) ) - \varphi^\prime ( u^N (t, x) ) ( \eta_1 (x) - u^N (t, x) ) \} dt \bigg| \bigg] \\
& \quad \le C \int_0^T H (\mu^N_t | \nu^N_t )dt   + O (K N^{d-\varepsilon_0}).
\end{aligned}
\]
as $ N $ tends to infinity. Hence we can approximate $I_1$ by $ \tilde{I}_1 \coloneqq \sum_{ x \in \mathbb{T}^d_N } \Delta^N \varphi ( u^N (t, x) ) \omega_{1, x} $ as 
\[
\mathbb{E}^N_{\mu^N_0 } \bigg[ \bigg| \int_0^T ( I_1 - \tilde{I}_1 ) dt  \bigg| \bigg] \le C K^3 e^{C K^2 } \int_0^T H (\mu^N_t | \nu^N_t ) dt + O ( K^4 e^{C K^2 } N^{d-\varepsilon_0 } )  
\]   
noting that the coefficient $\Delta^N \varphi (u^N (t, x) ) / \varphi (u^N (t, x) ) $ has order $ O (K^3 e^{C K^2 } ) $ in view of the $L^\infty$-estimate for second order discrete derivatives given in Lemma \ref{ggrad u} and the lower bound $\varphi (u^N (t,  x) )^{-1} \le C  e^{ CK} $ by Lemma \ref{unif est}. 
 
Next we show that $ I_3 $ can be decomposed into a linear part plus some residual term which can be quantitatively estimated. To see that, we take a local function $h$ as the summand appearing in the definition of $ I_3 $. Then, since $\eta_1 $ and $\eta_2 $ are independent under our reference measure $\nu_\beta$ for any $\beta = (\beta_1 , \beta_2 ) $, we calculate as  
\[
\begin{aligned}
 \tilde{h} (\beta) 
& = \bigg( \frac{\varphi (u^N (t, x ) ) }{ Z_{\beta_1} } \frac{ v^N (t, x ) }{1 -v^N (t, x ) } \sum_{ k= 0 }^\infty \frac{ k+1 }{ g (k+1) } \frac{ \varphi (\beta_1 )^k }{ g ( k ) ! } \bigg) (1- \beta_2 )   - \beta_1 \beta_2 \\
&  = \frac{ \varphi (u^N (t, x ) ) }{ \varphi (\beta_1) } \frac{ v^N (t, x ) }{1 -v^N (t, x ) }  \beta_1 ( 1-  \beta_2 ) -\beta_1 \beta_2 
\end{aligned}
\]
so that $\tilde{h} (u^N (t, x), v^N ( t, x) ) = 0$. Moreover, the above explicit calculation of $\tilde{h} (\beta)$ enables us to obtain
\[
\begin{aligned}
& \partial_1 \tilde{h} ( u^N (t,x) , v^N (t, x ) ) = - \frac{ \varphi^\prime (u^N (t, x ) ) }{ \varphi (u^N (t, x )  ) } u^N (t, x) v^N (t, x) , \\
& \partial_2 \tilde{h} ( u^N (t, x ) , v^N (t, x ) ) = - \frac{v^N ( t, x ) }{ 1- v^N (t, x ) } . 
\end{aligned}
\]
Therefore, applying the Boltzmann-Gibbs principle again for this local function $h$, there exists a positive constant $C$ such that 
\[
\begin{aligned}
 \mathbb{E}^N_{ \mu^N_0 } \bigg[  \bigg|   \int_0^T  ( I_3 - \tilde{I}_3 ) dt \bigg| \bigg] 
 \le C \int_0^T K H (\mu^N_t | \nu^N_t )dt + O (K^2 N^{d-\varepsilon_0})
\end{aligned}
\]
as $N$ tends to infinity where $\tilde{I}_3$ is defined by
\[
\tilde{ I }_3 = -  K \sum_{ x \in \mathbb{T}^d_N } u^N (t, x) v^N (t, x )  \omega_{1, x}  - K \sum_{ x \in \mathbb{T}^d_N } u^N (t, x) v^N (t, x) \omega_{2, x} .
\]
By this line, replace $I_1$(resp. $I_3 $) by $\tilde{I}_1$(resp. $\tilde{I}_3 $) and sum up all terms in \eqref{integrand terms} to see that linear terms in rescaled variables $\omega_1 $ and $\omega_2 $ cancel by virtue of the semi-discretized reaction-diffusion system \eqref{dHDL eq}.

Finally we conduct a replacement procedure for $I_2$. If we apply the Boltzmann-Gibbs principle also for $I_2 $, which is a two-point correlation of $\omega_2$, then we can replace it by a linear combination of $\omega_2 $. This object is expected to be asymptotically small since each $\omega_{2, x } $ is mean-zero under the local equilibrium state. To replace $I_2 $ exactly by zero for convenience for PDE part, we use the following Lemma \ref{Kawasaki}. Note that the object to be replaced here is $K^2 \sum_{ | x- y | = 1 } \omega_{ 2, x } \omega_{2, y } $ since we have a gradient estimate $| \nabla^N v^N (t, x ) | = O (\varepsilon^{-1/2 } K ) $ as $N$ tends to infinity by Lemma \ref{grad v}, while in \cite{FT18} the diffusion coefficient $\varepsilon $ for $v$ is a constant which is independent of the scaling parameter $N$.

\begin{lemma}
\label{Kawasaki}
Assume (A\ref{IF}), (B\ref{K}) and (B\ref{Epsilon}). Then for every $\alpha > 0$ and small $\kappa > 0$, there exist $C = C (\alpha ) > 0$ and some small $\varepsilon_0 > 0$ such that  
\[
\begin{aligned}
E_{ \mu^N_t } \bigg[ K^2 \sum_{ x, y \in \mathbb{T}^d_N , \, | x- y | = 1 } \omega_{2, x } \omega_{2, y } \bigg]
\le 
& \alpha \varepsilon  N^2 \mathcal{D}_K (\sqrt{ f^N_t  } ; \nu^N_t ) 
 + C K^2  H (\mu^N_t | \nu^N_t )  + O( N^{ d- \varepsilon_0 } ) 
\end{aligned}
\]
as $N$ tends to infinity. 
\end{lemma}

Now we apply this result for $ I_2 $ to obtain 
\[
E_{\mu^N_t } [ I_2 ] \le \varepsilon N^2 \mathcal{D}_K \big( \sqrt{f^N_t }; \nu^N_t \big) + C K^2 H (\mu^N_t | \nu^N_t )  + O ( N^{d- \varepsilon_0 } )
\]
for some small $\varepsilon_0 > 0 $.

\end{proof}

\subsection{Proof of Lemma \ref{Kawasaki}}
In this subsection, we prove Lemma \ref{Kawasaki}. The key idea of the proof is similar to previous researches (\cite{FT18}, \cite{DMFPV19} and \cite{H20}) so we only give a sketch here. Define  
\[
 V \coloneqq K^2 \sum_{ x \in \mathbb{T}^d_N } \omega_{2, x } \omega_{2, x + e_i } , \quad 
V^\ell \coloneqq K^2 \sum_{x \in \mathbb{T}^d_N } \overleftarrow{ ( \omega_2 ) }_{x,\ell} \overrightarrow{(\omega_2 )}_{ x + e_i , \ell }  
\]
for each $i = 1, \ldots , d $. Here $\{ e_j \}_{ j = 1, \ldots , d } $ is the normal basis of $\mathbb{Z}^d$ and we defined 
\[
\overleftarrow{G}_{x,\ell} \coloneqq \frac{1}{ | \Lambda^+_{ \ell } | } \sum_{y \in \Lambda^+_\ell} G_{ x - y }, \quad
\overrightarrow{G}_{x,\ell} \coloneqq \frac{1}{ | \Lambda^+_\ell | } \sum_{y \in \Lambda^+_\ell} G_{ x + y } 
\]
for $G=\{ G_x \}_{x \in \mathbb{T}^d_N}$ with $\Lambda^+_\ell \coloneqq [0,\ell-1] ^d  \cap \mathbb{Z}^d$. We omit the dependence on $ i $ for simplicity. One can estimate the cost to replace $V$ by its local average $V^\ell$ as follows.

\begin{lemma}
\label{replacement}
Assume (A\ref{IF}), (B\ref{K}) and (B\ref{Epsilon}). We choose $\ell= N^{1/d-\kappa/d}$ when $d \geq 2$ and $\ell= N^{1/2-\kappa}$ when $d=1$ with $\kappa>0$ sufficiently small. Then for any $\alpha > 0 $ there exists some $C= C (\alpha, \kappa ) > 0 $ such that
\[
E_{ \mu^N_t } [ V- V^\ell ] \le \alpha N^2 \mathcal{D}_K (\sqrt{ f^N_t } ; \nu^N_t) + C \left(H (\mu^N_t | \nu^N_t ) + N^{d-1+\kappa} \right)
\] 
when $d \geq 2$. When $d = 1$, we have the same estimate with the last term $N^{d-1+\kappa}$ replaced by $N^{1/2+\kappa}$. 
\end{lemma}

To give a proof of Lemma \ref{replacement}, we introduce the notion of flows on graph. 

\begin{definition}
Let $G = (V, E)$ be a finite graph where $V$ is a set of all vertices and $E$ is the set of all edges. For two probability measures $p,q$ on $V$, we call $\Phi= \{ \Phi(x,y) \}_{ \{x,y\} \in E }$ a flow on $G$ connecting $p$ and $q$ if it satisfies:
\begin{itemize}
\item $\Phi(y,x)=- \Phi(x,y)$ for all $\{ x,y \} \in E$,
\item $\sum_{z \in V } \Phi(x,z)= p(x) -q(x)$ holds for all $x \in V $.
\end{itemize}
\end{definition}

In the sequel, we regard any finite subset in $\mathbb{Z}^d$ as a graph where the set of all bonds means the set of all pair of two points in that set such that the Euclidean distance between them is $1$. The next result is a key ingredient. 

\begin{proposition}[Flow lemma, \cite{JM18}]
\label{flow lem}
Let $\delta_0$ be the Dirac measure on $\mathbb{Z}^d$ with mass $1$ on $0 \in \mathbb{Z}^d$ and let $ p_\ell $ be the uniform probability measure on $ \mathbb{Z}^d  $ with mass on $\Lambda_\ell $ defined by $p_\ell (x)= |\Lambda^+_\ell|^{-1} \mathbf{1}_{\Lambda^+_\ell}(x) $. Moreover, let $q_\ell $ be the probability measure on $\mathbb{Z}^d$ defined by $q_\ell (x) = p_\ell * p_\ell (x)  \coloneqq \sum_{ y \in \mathbb{Z}^d } p_\ell (y)  p_\ell (x-y) $. Then there exists a flow $\Phi^\ell$ on $\Lambda^+_{2\ell}$ connecting $\delta_0$ and $q_\ell$ such that $\Phi^\ell (x,y) = 0$ for any $x \in (\Lambda^+_{2\ell})^c$ and $y \in \mathbb{Z}^d$, and that
\[
\sum_{x \in \Lambda^+_{2\ell}} \sum_{ 1 \le j \le d} \Phi^\ell (x,x+e_j)^2  \le C_d g_d (\ell) , \quad 
g_d (\ell ) = 
\begin{cases}
\begin{aligned}
& \ell && \text{ if } d = 1 , \\
& \log \ell && \text{ if } d = 2 , \\
& 1&& \text{ otherwise. }
\end{aligned}
\end{cases}
\]
\end{proposition}



Taking a flow connecting $\delta_0$ and $q_\ell$, a careful calculation (see \cite{DMFPV19} or \cite{H20}) enables us to write  
\begin{equation}
\label{V-Vl}
V - V^\ell = K^2 \sum_{1 \le j \le d } \sum_{x \in \mathbb{T}^d_N}  h^{\ell, j }_x (\omega_{1,x} - \omega_{1, x+e_j} )
\end{equation}
where $h^{\ell ,j }_x $ is defined by 
$
h^{\ell, j}_{x} = \sum_{y \in \Lambda^+_{ 2 \ell } } \omega_{2, x- y } (\eta_1, \eta_2) \Phi^\ell(y,y+e_j) .
$
Note here that $h^{\ell ,j }_x $ is invariant under the transformation $\eta_1 \mapsto \eta^{ x, x + e_j } $ for any $y \in \Lambda^+_{ 2 \ell } $ and $j = 1,\ldots , d$. Moreover, since $\omega_{2, x} $ and $\omega_{2, x} \omega_{2, y} $ with $x \neq y$ has average zero under $\nu^N_t$, there exists a positive constant $C$ which is independent of $N$ such that 
\[
 E_{\nu^N_t } [ h^{\ell, j}_x ] = 0, \quad 
 \text{Var}_{\nu^N_t } [ h^{\ell, j }_x ] \le C g_d (\ell) e^{ 2 C_1 K }
\]     
by the flow lemma (Proposition \ref{flow lem}) and the lower bound of $u^N$ according to Lemma \ref{unif est}.

\begin{lemma}
\label{IBP2}
Assume (A\ref{IF}) and let $f$ be any density with respect to $\nu^N_t $. Then we have  
\begin{equation}
\label{summand V-Vl}
\int_{\mathcal{X}_{ N} } h^{\ell, j }_x  (\omega_{ 2 ,x + e_j } - \omega_{2,x } ) f d\nu^N_t = 
\int_{\mathcal{X}_{ N} } h^{\ell ,j }_x  \frac{ \eta_{ 2 , z } }{ \chi_2 (v^N ( t,  x ) ) } \big[ f (\eta_2^{x , x + e_j} ) - f ( \eta_2 ) \big] d\nu^N_t + R_{x ,j }
\end{equation}
for every $x \in \mathbb{T}^d_N$ and $ j= 1, \ldots , d$. Moreover, the error term $R_{x, j } $ is bounded as 
\begin{equation}
\label{R bound}
|R_{ x, j }| \le  C e^{3C_1 K} |v^N ( t, x )- v^N (t,  x + e_j)| \int_{\mathcal{X}_{ N} } |h^{\ell,j}_x (\eta)| f d\nu^N_t . 
\end{equation}
\end{lemma}

We next bound the summand in (\ref{V-Vl}) with the help of the above results. 
$\mathcal{D}_K (f ; \nu) =  \sum_{ x \in \mathbb{T}^d_N } \sum_{ j = 1 , \ldots , d } \mathcal{D}_{K ; x, x+e_j }( f ; \nu)$.

\begin{lemma}
\label{summand est}
Assume (A\ref{IF}) and let $f$ be any density with respect to $\nu^N_t$. Then there exists a positive constant $C$ such that for every $\beta>0$, $x \in \mathbb{T}^d_N$ and $j = 1, \ldots  , d $, 
\[
\int_{\mathcal{X}_N} h^{\ell,j}_x (\omega_{2, x } - \omega_{2, x + e_j}) f  d\nu^N_t 
\le \beta \mathcal{D}_{K ; x , x + e_j } (\sqrt{ f }; \nu^N_t ) 
+ \frac{C}{\beta}  e^{ 3 C_1 K} \int_{\mathcal{X}_N} (h^{\ell,j}_x)^2  f d\nu^N_t  + R_{ x, j }
\]
and each error term $R_{ x, j }$ satisfies the bound \eqref{R bound}. Here, 
\[
\mathcal{D}_{K; x,x+e_j}( f ; \nu)  \coloneqq \frac{ 1 }{ 2 } \int_{\mathcal{X}_N} { \big[ f (\eta_1 , \eta_2^{x,x+e_j} ) - f(\eta_1, \eta_2)  \big] }^2 d\nu(\eta)
\]
\end{lemma}

\begin{proof}[Proof of Lemma \ref{replacement}]
Recalling the representation of $V-V^\ell$ in (\ref{V-Vl}), we estimate  
\[
E_{ \mu^N_t } [ V- V^\ell ] 
= K^2 \sum_{ 1 \le j \le d } \sum_{x \in \mathbb{T}^d_N}  \int_{\mathcal{X}_N}  h^{\ell,j}_x (\omega_{2,x}- \omega_{2,x+e_j}) f^N_t d\nu^N_t .
\]
We apply Lemma \ref{summand est} for $f = f^N_t = d\mu^N_t / d \nu^N_t $, taking $\beta = \alpha \varepsilon N^2 K^{-2}$ with $\alpha >0$. Then the above quantity is bounded above by
\[
\begin{aligned}
\alpha \varepsilon N^2 \mathcal{D}_K (\sqrt{f^N_t }; \nu^N_t ) 
+ \frac{C K^4 }{\alpha \varepsilon N^2} e^{3 C_1 K} \sum_{ 1 \le j \le d} \sum_{x \in \mathbb{T}^d_N}  \int_{\mathcal{X}_N}  (h^{\ell,j}_x)^2 d\mu^N_t  
+ K^2 \sum_{1 \le j \le d } \sum_{x \in \mathbb{T}^d_N} R_{ x, j }.
\end{aligned}
\]
Recall that the residual term $R_{x, j } $ has the bound (\ref{R bound}) for every $x \in \mathbb{T}^d_N$ and $j=1,\ldots , d$. Since $| v^N (t, x)- v^N (t, x+e_j ) | \le C \varepsilon^{-1/2 } K N^{-1}  $ by the gradient estimate given in Lemma \ref{grad u}, estimating $| h^{\ell,j}_x | \leq  1 + (h^{\ell,j}_x)^2 $, we have 
\[
K^2 |R_{1,x,j} | \le C \varepsilon^{-1/ 2 } K^3 N^{-1} e^{3 C_1 K} \int_{\mathcal{X}_N} \big( 1+ (h^{\ell,j}_x)^2 \big) d\mu^N_t 
\]
for every $x \in \mathbb{T}^d_N $ and $j = 1, \ldots , d $. Therefore, the expectation with respect to $\mu^N$ of $V - V^\ell$ is bounded above by
\[
\alpha \varepsilon  N^2 \mathcal{D}_K (\sqrt{ f^N_t }; \nu^N_t ) 
+ C \varepsilon^{- 1 } K^4 N^{-1 } e^{3 C_1 K} \sum_{1 \le j \le d } \sum_{x \in \mathbb{T}^d_N} \int_{\mathcal{X}_N} (h^{\ell,j}_x)^2 d\mu^N_t 
+ C \varepsilon^{-1 /2 } K^3 e^{3 C_1 K } N^{d-1}
\]
with some positive constant $C = C (\alpha ) $. For the second term, noting that the random variables $\{ h^{\ell,j}_x \}$ are $(2 \ell-1)$-dependent, we decompose the summation $\sum_{x \in \mathbb{T}^d_N}$ into $\sum_{y \in \Lambda^+_{2\ell} } \sum_{z \in (4 \ell) \mathbb{T}^d_N \cap \mathbb{T}^d_N}$ and then apply the entropy inequality. Then we have   
\[
\begin{aligned}
\sum_{x \in \mathbb{T}^d_N} \int_{\mathcal{X}_N} (h^{\ell,j}_x)^2 d\mu^N_t  
& \le  \gamma^{-1} \sum_{y \in \Lambda^+_{2 \ell}} 
\bigg( H(\mu^N_t | \nu^N_t ) + \log{  \int_{\mathcal{X}_N} \prod_{z \in (4 \ell) \mathbb{T}^d_N \cap \mathbb{T}^d_N} e^{\gamma (h^{\ell,j}_{z+y})^2} d\nu^N_t  }   \bigg)   \\
& = \gamma^{-1 } (2\ell)^d  \bigg( H ( \mu^N_t | \nu^N_t )  + \sum_{z \in (4\ell) \mathbb{T}^d_N \cap \mathbb{T}^d_N  } \log{ \int_{\mathcal{X}_N} e^{\gamma (h^{\ell,j}_{z+y})^2 } }  d \nu^N_t \bigg)
\end{aligned}
\]
for every $\gamma >0$. Moreover, recall here that by the flow lemma stated in Proposition \ref{flow lem} we can estimate the variance of $h^{\ell,j}_x$ as 
\[
\sigma^2 \coloneqq \sup_{x \in \mathbb{T}^d_N,  j = 1,\ldots ,d } {\rm Var}_{ \nu^N_t } [ h^{\ell,j}_x ]  \leq C_d g_d (\ell) e^{ 2 C_1 K }
\]
with $g_d(\ell)$ in Proposition \ref{flow lem}. Therefore, applying the concentration inequality, 
\[
\log \int_{\mathcal{X}_N} e^{\gamma (h^{\ell,j}_x)^2} d\nu^N_t \leq  2
\]
for every $0 < \gamma \leq  C_0 \sigma^{-2}$. Hence, choosing $\gamma^{-1} = C_0^{-1} C_d g_d(\ell) e^{ 2 C_1 K } $, we can see that $E_{\mu^N_t } [ V -V^\ell  ] $ is bounded above by
\[
\alpha \varepsilon N^2 \mathcal{D}_K ( \sqrt{ f^N_t  }; \nu^N_t ) 
+ C \varepsilon^{- 1 } \ell^d g_d(\ell) K^4 e^{ 5 C_1 K}  N^{-1 } \big( H ( \mu^N_t  | \nu^N_t ) + N^d \ell^{-d }  \big)  
+ C \varepsilon^{- 1/ 2 } K^3 e^{ 3 C_1 K } N^{d-1}
.\] 
Now recall that the assumptions (B\ref{K}) and (B\ref{Epsilon}) assured the growth rate of $K$ and $\varepsilon^{-1 } $ was suppressed by a monomial of $N$ with a sufficiently small exponent. In particular, divergent elements in the above display along the parameters $K $ and $\varepsilon^{-1} $ can be absorbed into $N$ in the denominators so that we may treat $K$ and $\varepsilon^{-1 }$ as if they were constants independent of $N$. Hence we complete the proof by choosing $\ell = N^{ 1/ d  - \kappa / d }$ when $d \ge 2$ and $\ell = N^{ 1/2-\kappa }$ when $d = 1$. 
\end{proof}

\begin{lemma}
Choose $\ell = N^{ 1/ d - \kappa/ d } $ when $d \ge 2 $ and $\ell = N^{ 1/2 - \kappa } $ when $d = 1 $. For any $\kappa > 0$, there exists a positive constant $C$ such that 
\[
E_{\mu^N_t} [V^\ell ]  \le C K^2 \big( H ( \mu^N_t | \nu^N_t ) + N^{d-1+\kappa} \big) 
\]
when $d \geq 2$. When $d = 1$, the last term on the right hand side of the above is replaced by $N^{1/2 + \kappa}$. 
\end{lemma}

\section{Several estimates for the reaction-diffusion system \eqref{dHDL eq}}
\label{sec:PDE estimates}

\begin{lemma}[Comparison principle]
\label{comparison}
Let $\varphi $ be a differentiable non-decreasing function on $\mathbb{R}$ and let $(t, u ) \mapsto f (t, x, u ) $ be a $C^1$-smooth function on $[0, T] \times \mathbb{R}^{ N^d }$ for every $x \in \mathbb{T}^d_N$. Let $u^N = \{ u^N (t, x ) \}_{  t \ge 0, x \in \mathbb{T}^d_N }$ be a solution of 
\begin{equation}
\label{comparison eq}
\partial_t u^N(t, x) = \Delta^N \varphi (u^N (t, x ) ) + f (t, x, u^N (t) )
\end{equation}
and let $u^N_+$ (resp. $u^N_-$) be super- (resp. sub-) solution of \eqref{comparison eq}, namely, $u^N_+$ (resp. $u^N_-$) satisfies the equation \eqref{comparison eq} with ``$\ge$'' (resp. ``$\le$'') instead of the equality. Assume that $u^N_+ (0, x ) \ge u^N (t, x ) $ (resp. $u^N_- (0, x ) \le u^N (0, x) $) holds for every $x \in \mathbb{T}^d_N$. Then we have that $u^N_+ (t, x ) \ge u^N (t, x )$ (resp. $u^N_-(t, x ) \le u^N (t, x) $) holds for every $t \in [0,T]$ and $x \in \mathbb{T}^d_N$. 
\end{lemma}

The proof is intuitively given as follows. We may only consider super-solutions since sub-solutions can be treated in the same way. Suppose there exists some point $(t_0, x_0 ) $ such that $u^N_+ (t_0, x_0) = u^N (t_0, x_0)$. Then, since $u^N_+$ is a super-solution of \eqref{comparison eq} and $\varphi $ is an increasing function, we have  
\[
\begin{aligned}
& \partial_t (u^N_+ - u^N ) (t_0, x_0)  \\
&\quad  \ge \Delta^N \big( \varphi (u^N_+) - \varphi ( u^N ) \big) (t_0, x_0) + \big( f (t_0, x_0, u^N_+ (t_0) ) - f (t_0, x_0, u^N (t_0) )  \big) \\
& \quad = \sum_{ j = 1 }^{ 2d } \big( \varphi (u^N_+) - \varphi (u^N ) \big) (t_0, x_0 + e_j )  \ge 0
\end{aligned}
\]
where we set $e_{j + d } = - e_j$ for every $j = 1,\ldots , d $. The above estimate implies that $u^N_+ - u^N $ is positively drifted at the point $(t_0, x_0)$ and thus the solution $u^N$ will not be able to exceed the super-solution $u^N_+$ for any time $t \in [0,T]$. 

\begin{proof}[Proof of Lemma \ref{comparison}]
We show the assertion only for super-solutions since sub-solutions can be treated similarly. Let $u^N_+ $ be a super--solution of \eqref{comparison eq} satisfying $u^N_+ (0, x ) \ge u^N (0, x )$ for every $x \in \mathbb{T}^d_N$. Then, by definition of super-solution, $u^N_+$ satisfies 
\[
\partial_t \overline{u}^N (t,x) \ge \Delta^N \overline{u}^N(t,x) + f (t, x, \overline{u}^N(t) )
\] 
for every $t \in [0,T]$ and $x \in \mathbb{T}^d_N $. Then subtracting \eqref{comparison eq} on the above display to get 
\begin{equation}
\begin{aligned}
\label{difference}
\partial_t ( u^N_+ - u^N ) (t,x)  \ge 
 \Delta^N ( u^N_+ - u^N ) ( t,x)   
 + \tilde{f}( t, x, u_+ (t) , u (t) ) ( u^N_+ - u^N ) (t,x)   .
\end{aligned}
\end{equation}
Here, $\tilde{ f } = \tilde{f} (t, x, u^N_+ (t) , u^N (t) )$ is defined by 
\[
\tilde{f} (t, x, u^N_+ (t) , u^N (t) ) = 
\begin{cases}
 \displaystyle\frac{ f ( t,  x, u^N_+ (t ) ) - f (t, x, u^N (t)) }{ u^N_+ (t,x) - u^N (t,x) } &\text{if } u^N_+ (t,x)  \neq u^N (t,x) ,  \\
\displaystyle\frac{ \partial f }{  \partial u (x) } ( t, x, u  ) \bigg|_{ u  = u^N (t)  }   &\text{if } u^N_+ (t,x) = u^N (t,x) .
\end{cases}
\]
Let $M \coloneqq \sup_{ (t, x ) \in [0, T] \times \mathbb{T}^d_N }  | \tilde{ f } (t, x, u^N_+ (t) , u^N (t ) ) | $ and let $w^N (t, x ) $ be a function defined by 
\[
w^N (t,x) = ( u^N_+ (t,x) - u^N (t,x) ) e^{ M t } + 2 \varepsilon - \varepsilon e^{-t}
\]
for any $\varepsilon > 0$. Note here that such $M < \infty$ exists since $u^N_+ (t, x)$ and $u^N (t, x)$ are both continuous in $t$ for every $x \in \mathbb{T}^d_N$. Moreover, by the assumption for the initial function we have $w^N (0, x) > 0$ for every $x \in \mathbb{T}^d_N$. In the sequel, we prove $w^N \ge 0 $ in $[0,T] \times \mathbb{T}^d_N$ by showing contradiction. Suppose there exists a point $(t_0, x_0 ) \in (0,T] \times \mathbb{T}^d_N$ such that $w^N(t_0, x_0) = 0$ for the first time and $w^N (t,x) > 0$ for every $t \in [0, t_0)$ and $x \in \mathbb{T}^d_N$. Then, since $(t_0, x_0 )$ attains minimum of $w^N$ in $[0, t_0] \times \mathbb{T}^d_N$, we have $\partial_t w^N (t_0, x_0) \le  0$ and $\Delta^N w^N (t_0, x_0) = \Delta^N (u^N_+ - u^N) (t_0, x_0) e^{M t_0 } \ge 0$. In particular, since $\varphi (u^N_+ (t_0, x_0 )) - \varphi (u^N (t_0, x_0 ) ) = \varphi^\prime (\hat{u} ) (u^N_+ - u^N ) (t_0, x_0 ) $ for some $\hat{u} $ between $u^N_+ (t_0, x_0) $ and $u^N (t_0, x_0 ) $ by the mean-value theorem, recalling the definition of the discrete Laplacian $\Delta^N $ and that $\varphi $ is non-decreasing, we deduce $\Delta^N \big( \varphi (u^N_+ (t_0, x_0 )) - \varphi (u^N (t_0, x_0 ) ) \big) e^{M t_0 } \ge 0$. Therefore, we have 
\begin{equation}
\label{difference 1}
\partial_t w^N (t_0, x_0 ) - \Delta^N \big( \varphi (u^N_+ (t_0, x_0 )) - \varphi (u^N (t_0, x_0 ) ) \big) e^{M t_0 } \le 0 .
\end{equation} 
On the other hand, letting $\tilde{u}^N \coloneqq u^N_+ - u^N$, we have 
\[
\begin{aligned}
& \partial_t w^N (t_0, x_0) - \Delta^N \big( \varphi (u^N_+ (t_0, x_0 )) - \varphi (u^N (t_0, x_0 ) ) \big)   \\
& \, = \big(  \partial_t  \tilde{u}^N (t_0, x_0)  + M \tilde{u}^N (t_0, x_0) \big) e^{ M t_0 } + \varepsilon e^{-t_0 } 
- \big( \varphi (u^N_+ (t_0, x_0 )) - \varphi (u^N (t_0, x_0 ) ) \big) e^{M t_0 }   \\
& \, \ge \big(   - \tilde{f} (t_0,  x_0 , \overline{u}^N (t_0) , u^N (t_0) ) + M \big) \tilde{u}^N (t_0, x_0) e^{ M t_0 } + \varepsilon e^{-t_0 } 
\end{aligned}
\]
since $u^N_+ $ is a super-solution. However, recalling the definition of $M$, the last quantity is bounded from below by a strictly positive constant, which is contradiction to \eqref{difference 1}. Therefore, we have $w^N \ge 0$ in $[0, T] \times \mathbb{T}^d_N $ so that $u^N_+ (t, x ) - u^N (t, x )  \ge \varepsilon ( e^{-t } - 2 ) e^ {- M t }$ for every $(t, x ) \in [0, T ] \times \mathbb{T}^d_N  $. Since $\varepsilon > 0$ was taken arbitrary, we complete the proof by letting $\varepsilon$ tends to zero. 
\end{proof}

In the sequel, let $u^N $ and $v^N$ be a solution of semi-discretized system \eqref{dHDL eq}. Applying the above comparison theorem, we get the following uniform estimates for values of $u^N$ and $v^N$ which is independent of $N$.

\begin{lemma}
\label{unif est}
Assume there exist constants $C_1 > 0 $, $M_u \ge 0 $ and $0 \le M_v \le 1 $ such that 
\[
e^{ - C_1 K } \le u^N ( 0, x ) \le M_u , \quad 
e^{ - C_1 K } \le v^N (0, x ) \le M_v
\]
for every $x\in \mathbb{T}^d_N $. Then, there exists a constant $C = C ( C_1, T ) > 0 $ such that 
\[
e^{- C K  } \le u^N (t, x)  \le M_u , \quad 
e^{- C K  } \le v^N (t, x)  \le M_v  
\]
for every $t \in [0, T] $ and $x \in \mathbb{T}^d_N $. In particular, $(u^N, v^N )$ takes values in $\mathbb{R}_+ \times [0, 1]$.  
\end{lemma}
\begin{proof}
This can be done due to the comparison theorem given above. First, we can easily see that the zero-function is a sub-solution for both two equations of the system \eqref{dHDL eq} so that we may assume $u^N$ and $v^N$ are non-negative. Then, we see that the constant function $M_u $ (resp. $M_v $) is a super-solution for the first (resp. second) equation of the system \eqref{dHDL eq} and the upper bound is thus proved. Next we show the lower bound. Let us define $\underline{u}^N (t,  x ) =e^{-C_1 K } e^{ - M_v K t } $ and $\underline{v}^N (t, x ) =e^{- C_1 K } e^{- M_u K t } $. Then by the assumption, we have $u^N (0, x ) \ge \underline{u}^N (0, x ) = e^{ - C_1 K } $ and $v^N (0, x ) \ge \underline{v}^N (0, x) = e^{ - C_1 K } $ for every $x \in \mathbb{T}^d_N$. Moreover, since $\underline{u}^N $ and $\underline{v}^N $ are spatially homogeneous, we have 
\[
\begin{aligned}
& \partial_t \underline{ u }^N (t, x) =\Delta^N \varphi (\underline{u}^N (t, x ) ) - K M_v \underline{u}^N (t, x) 
\le \Delta^N \varphi (\underline{u}^N (t, x ) )  - K v^N (t, x ) \underline{u}^N (t, x ),  \\
& \partial_t \underline{ v }^N (t, x) = \varepsilon \Delta^N \underline{ v }^N (t, x )  - K M_u \underline{ v }^N (t, x) 
\le \varepsilon \Delta^N \underline{u}^N (t, x )   - K u^N (t, x ) \underline{v}^N (t, x )  
\end{aligned}
\]
for every $t \in [0, T]$ and $x \in \mathbb{T}^d_N$ noting $u^N$ and $v^N$ stay non-negative. Therefore, $ \underline{u}^N $ and $ \underline{v}^N  $ are sub-solutions and thus we obtain the desired lower bound again by the comparison principle (Theorem \ref{comparison}). 
\end{proof}


\begin{lemma}
\label{react int}
We have that
\[
\sup_{N\in \mathbb{N}} \int_0^T  \frac{1}{N^d} \sum_{x \in \mathbb{T}^d_N} K(N) u^N(t,x) v^N(t,x) dt 
\le M_u 
.\]
\end{lemma}
\begin{proof}
From the first equation of (\ref{dHDL eq}), integrating over $t \in [0,T]$ and $x \in \mathbb{T}^d_N$ to represent the integration of reaction term by terms which are independent of $K(N)$. Since summation over $x \in \mathbb{T}^d_N$ of $\Delta^N u^N(t,x)$ vanishes and the term involving time derivative becomes an integration on the boundary, the proof is obvious in view of the uniform boundedness of $u^N $ (Lemma \ref{unif est}).
\end{proof}

Recall for every $\ell \in \mathbb{N}$ we denote $\Lambda_\ell \coloneqq [-\ell,\ell ]^d \cap \mathbb{Z}^d $ the centered rectangle with length $2 \ell + 1$ and we write its volume as $| \Lambda_\ell | = (2\ell + 1 )^d $. We have the following energy estimates for semi-discretized solutions $u^N$ and $v^N$.

\begin{lemma}
\label{energy est}
Assume $u^N$ and $v^N $ are bounded non-negative solution of \eqref{dHDL eq}. Then there exists a positive constant $C$ such that 
\[
\int_0^T \sum_{ x \in \mathbb{T}^d_N } | \nabla^N u^N (t, x ) |^2  dt \le C N^d , \quad 
\int_0^T \sum_{x \in \mathbb{T}^d_N } \varepsilon (N) | \nabla^N v^N (t, x ) |^2 dt \le  C N^d . 
\]
Moreover, we have 
\[
\int_0^T \sum_{ x \in \mathbb{T}^d_N } \bigg( \frac{1}{ | \Lambda_\ell | } \sum_{ z \in \Lambda_\ell } \big( u^N (t, x + z ) - u^N (t, x ) \big) \bigg)^2 dt \le C N^{d-2} \ell^2 .
\]
\end{lemma}
\begin{proof}
According to the first equation of the semi-discretized reaction-diffusion system \eqref{dHDL eq}, the summation by parts formula gives 
\[
\begin{aligned}
\frac{1}{2} \frac{\partial}{ \partial t } \sum_{ x \in \mathbb{T}^d_N } u^N (t, x)^2 
& = \sum_{x \in \mathbb{T}^d_N } u^N (t, x )  \big( \Delta^N \varphi ( u^N (t, x ) ) - K u^N (t, x ) v^N (t, x )  \big) \\
& = \sum_{x \in \mathbb{T}^d_N }  \big( - \nabla^N u^N (t, x ) \cdot \nabla^N \varphi ( u^N (t, x ) ) - K u^N (t, x )^2 v^N (t, x )  \big) \\
& \le  - c_0 \sum_{ x \in \mathbb{T}^d_N } | \nabla^N u^N ( t, x ) |^2
\end{aligned}
\]
with $c_0 > 0$. Here in the last estimate we used the mean-value theorem (note here that the function $\varphi$ is assumed to be differentiable and non-decreasing) and dropped the second term in the penultimate line. Therefore, noting the boundedness of $u^N$, integrate over time argument $t \in [0,T]$ to get the desired energy estimate. The assertion for $v^N$ can be shown similarly. 

On the other hand, by Jensen's inequality we have 
\[
\begin{aligned}
\bigg( \frac{1}{ | \Lambda_\ell | } \sum_{ z \in \Lambda_\ell } \big( u^N (t, z ) - u^N (t, x) \big) \bigg)^2 
& \le \frac{1}{ | \Lambda_\ell | } \sum_{ z \in \Lambda_\ell } \big( u^N (t, z ) - u^N (t, x) \big)^2 . 
\end{aligned}
\]
For every $z \in \mathbb{T}^d_N$, we can take a shortest path $ y_0, ... , y_{ |z| } $ from $0 $ to $z$, namely, $y_0 = 0$, $y_{ |z| } = z $ and $| y_{i+1 } - y_i | = 1$ for each $i = 0,...,|z|-1 $. Then, decompose $u^N (t, x+z) - u^N (t, x) = [ u^N (t, x + z  ) - u^N (t, x + y_{ |z| -1 } ) ] + \cdots + [ u^N (t, x + y_1 ) - u^N (t, x ) ] $ and use $(a_1 + \cdots + a_n)^2 \le n (a_1^2 + \cdots + a_n^2)$ to conclude 
\[
\begin{aligned}
 N^2 \sum_{x \in \mathbb{T}^d_N } \sum_{z \in \Lambda_\ell } \big( u^N (t, x+z) - u^N (t,  x) \big)^2 
& \le \sum_{ x \in \mathbb{T}^d_N } \sum_{ z \in \Lambda_\ell } |z| \sum_{ i=0 }^{ | z |-1 }   | \nabla^N u^N (t, x + y_i )  |^2 \\
& =  \sum_{ z \in \Lambda_\ell } |z|^2 \sum_{ x \in \mathbb{T}^d_N }  | \nabla^N u^N (t, x ) |^2  .
\end{aligned}
\]
Now integrating over $t \in [0,T]$ and using the first assertion which is proved in the above, the last quantity is bounded above by $C \ell^{d + 2} N^d$ for some $C > 0$. Hence, dividing by $( 2 \ell +1 )^d $, we get the second assertion and complete the proof. 
\end{proof}

Next we give some $L^\infty$-estimates for discrete gradient for $u^N$ and $v^N$. To see that, it is crucial to estimate discrete derivatives of the fundamental solution of divergence operators corresponding to $u^N$ and $v^N$, for which the next result by \cite{DD05} is applicable.

\begin{proposition}[\cite{DD05}]
\label{heat ker}
Let $p (t, x, y ) $ be the fundamental solution corresponding to a symmetric, uniform elliptic divergence operator $\mathcal{L}$ on $\mathbb{Z}^d$. Then there exist positive constants $C$ and $c $ such that 
\[
| \nabla^N p (t, x, y ) | \le C N (t \vee 1 )^{- 1/ 2 } p (ct , x ,y ) 
\]
for every $t > 0$ and $x,y \in \mathbb{T}^d_N $. 
\end{proposition}

Then, as indicated in \cite{FT18} or \cite{DMFPV19}, letting $p^N (t, x, y) $ be the fundamental solution corresponding to a symmetric, uniform elliptic divergence operator $\mathcal{L}^N $ on $\mathbb{T}^d_N $, we have $ p^N (t, x, y) = \sum_{ k \in N \mathbb{Z}^d } p (N^2 t , x, y + k ) $ for every $t \ge 0  $ and $x , y \in \mathbb{T}^d_N$. Therefore, applying the result by \cite{DD05} (Proposition \ref{heat ker}), we obtain an estimate 
\begin{equation}
\label{fundamental}
 | \nabla^N p^N (t, x, y ) | \le C t^{ - 1/2  } p^N (ct , x, y ) .
\end{equation} 
Now by applying this $L^\infty$-estimate for discrete gradient of the fundamental solution, we can also estimate discrete derivatives of $u^N$ and $v^N$. The proof is similar to \cite{EFHPS20}, which concerns a fast-reaction limit problem for one-species model.

\begin{lemma}
\label{grad u}
There exists a positive constant $C $ such that 
\[
| \nabla^N u^N (t, x ) | \le  K (C_0 + C t^{1/2}  ) 
\]
for every $t \in [0, T ] $ and $x \in \mathbb{T}^d_N $ provided the assertion holds at time $t = 0 $. 
\end{lemma}

\begin{proof}
First we take $ j $-th discrete derivative $\partial^N_j$ to the first equation of \eqref{dHDL eq} to deduce 
\begin{equation}
\label{RD derivative}
\partial_t \partial^N_j u^N (t, x) = \Delta^N \partial^N_j \varphi ( u^N (t, x) ) 
- K \partial^N_j \big( u^N (t, x ) v^N (t, x ) \big). 
\end{equation}
The first term in the right-hand side of the last identity can further be calculated as 
\[
\begin{aligned}
& \Delta^N \partial^N_j \varphi (u^N (t, x ) ) \\
& = N \nabla^N \cdot \tau_{-e } \nabla^N \big( \varphi (u^N (t, x + e_j ) ) - \varphi ( u^N (t, x ) )  \big) \\
& = N^2 \nabla^N \cdot \big[ 
\big( \varphi (u^N (t, x + e_j ) ) - \varphi (u^N (t, x + e_j - e_i  ) ) \big) 
- \big( \varphi (u^N (t,x) ) - \varphi (u^N (t, x -e_i ) ) \big) 
\big]_{ i =1,\ldots , d  } . 
\end{aligned}
\]
However, by Taylor's theorem, there exist $ u_1 $ between $u^N (t, x + e_j ) $ and $u^N (t, x )$, and $u_2$ between $u^N (t, x + e_j -e_i ) $ and $u^N (t, x ) $ such that 
\[
\begin{aligned}
\varphi (u^N (t, x + e_j ) ) - \varphi (u^N (t, x + e_j - e_i  ) )  =
 & \partial_j \varphi (u^N (t, x ) ) \big( u^N (t, x + e_j )  - u^N (t, x + e_j - e_i ) \big) \\
&  + \frac{ \varphi^{\prime \prime} (u_1 ) }{2 } \big( u^N (t, x + e_j ) - u^N (t, x ) \big)^2 \\ 
&  + \frac{ \varphi^{\prime \prime} (u_2 ) }{2 } \big( u^N (t, x + e_j - e_i ) - u^N (t, x ) \big)^2 .
\end{aligned}
\]
Similarly, we also have an expansion  
\[
\begin{aligned}
\varphi (u^N (t, x  ) ) - \varphi (u^N (t, x  - e_i  ) )  =
 & \partial_j \varphi (u^N (t, x ) ) \big( u^N (t, x )  - u^N (t, x - e_i ) \big) \\
&  + \frac{ \varphi^{\prime \prime} (u_1 ) }{2} \big( u^N (t, x ) - u^N (t, x - e_i ) \big)^2  
\end{aligned}
\]
for some $u_3 $ between $u^N (t, x -e_i) $ and $u^N (t, x ) $. Hence we can rewrite \eqref{RD derivative} as
\begin{equation}
\label{RD derivative2}
\partial_t \partial^N_j u^N (t, x ) = \mathcal{L}^N \partial^N_j u^N (t, x ) + R_\varphi (t, x ) 
- K \partial^N_j \big( u^N (t, x ) v^N (t, x ) \big) 
\end{equation}
where $\mathcal{L}^N_1 $ is a second order discrete divergence operator defined by
\[
\mathcal{L}^N_1 w (x) = \nabla^N \cdot\varphi^\prime (u^N (t, x ) )  \tau_{ -e } \nabla^N w (x)  
\]  
for any real-valued function $w$ on the configuration space $\mathcal{X}_N$. Regarding $u^N$ as a given function in the above definition, we can easily see that the operator $\mathcal{L}^N_1 $ is a symmetric, non-positive linear functional on $\ell^2 (\mathbb{T}^d_N ) $. Therefore, there exists a fundamental solution corresponding to the operator $ \mathcal{L}^N $ on $\mathbb{T}^d_N $, which is denoted by $p^N_1 (t, x, y ) $. Hence we apply Duhamel's principle to the above identity \eqref{RD derivative2} to obtain 
\[
\begin{aligned}
\partial^N_j u^N (t, x) 
& = \sum_{y \in \mathbb{T}^d_N } p^N_1 (t, x , y ) \partial^N_j u^N (0, x) + \int_0^t \sum_{ y \in \mathbb{T}^d_N } p^N_1 (t-s ,x, y ) R_\varphi (s, y) ds  \\
& \quad - K \int_0^t \sum_{ y \in \mathbb{T}^d_N } p^N_1 (t - s , x , y ) \partial^N_j \big( u^N (s, y ) v^N (s, y ) \big) ds  
\end{aligned}
\]
for every $j =1,\ldots , d  $. Recall here that $R_\varphi $ is a divergence form of quadratic function of $u^N$ and also note that the third term in the last identity is a divergence form. Therefore, integration by parts combining with an estimate \eqref{fundamental} for $p^N $ and the uniform boundedness of $u^N $ and $v^N$ given in Lemma \ref{unif est} enables us to deduce 
\[
\| \nabla^N u^N (t, \cdot ) \|_{L^\infty } \le C \| \nabla^N u^N (0, \cdot ) \|_{L^\infty } 
+ C \int_0^t (t-s )^{-1/2 } \| \nabla^N u^N (s, \cdot ) \|^2_{L^\infty } ds   
\]
for some positive constant $C$. Now let $m (t ) \coloneqq \sup_{ 0 \le s \le t } \sup_{ x \in \mathbb{T}^d_N} | \nabla^N u^N (s, x ) | $. Then combining with the assumption for $\nabla^N u^N (0, x) $, there exists a constant $C > 0 $ such that 
\[
m (t) \le  C \sqrt{t} m(t)^2 + C K 
\] 
for every $t \in [0, T] $. By solving this quadratic inequality, we have 
\[
m ( t ) \le  \frac{ 1- \sqrt{ 1- 4 C^2  K \sqrt{ t } } }{ 2 C \sqrt{t} }
\]
provided $t \le t_* \coloneqq  (4 C^2 K )^{-2 } $ where the other inequality can be rejected due to an estimate around $t = 0$ by the assumption for the initial value $m (0) $ and the continuity of $ t \mapsto m(t) $. Since the right-hand side of the last estimate is increasing in time $t$, we obtain $m(t) \le 2 C K $ for every $t \le t_*$.

On the other hand, we derive an estimate for the other range $ ( t_* , T ] $ of the interval $[0, T]$. To see that, one can notice that the non-linear diffusion operator $ \Delta^N \varphi (\cdot ) $ can be regarded as a linear Laplacian. Indeed, we have for any real-valued function $u $ on $ \mathbb{T}^d_N $ that 
\[
\Delta^N \varphi (u (x) ) 
= \nabla^N \cdot \tau_{- e }  \nabla^N \varphi (u (x) )  
= \nabla^N \cdot \tau_{- e } \big(  \varphi^\prime (x, e_j ; u ) \partial^N_j u (x)  \big)_{ j = 1,\ldots , d } 
\]
where $\tau_{ -e } $ denotes the shift to the direction $ (-1,\ldots , -1) \in \mathbb{R}^d $ acting on any $d$-dimensional vectors, and $\varphi^\prime (x, e_j ; u )$ is defined by 
\[
\varphi^\prime (x , e_j ; u )  = 
\begin{cases}
\begin{aligned}
& \displaystyle\frac{ \varphi (u (x + e_j )) - \varphi (u (x) ) }{ u (x + e_j ) - u (x) }  && \text{ if } u (x + e_j ) \neq  u (x) , \\ 
&  \varphi^\prime (u (x) ) && \text{ if } u (x + e_j) = u (x)  . 
\end{aligned}
\end{cases}
\] 
Now let $\mathcal{L}^N_2 $ be a second order discrete divergence operator defined by 
\[
\mathcal{L}^N_2 w (x) = \nabla^N \cdot \tau_{-e } \big( \varphi^\prime (x, e_j ; u^N (t ) ) \partial^N_j w (x) \big)_{ j = 1,\ldots , d }
\]
for every real-valued function $w$ on the configuration space $\mathcal{X}_N$. Similarly for the operator $\mathcal{L}^N_1 $ given above, the operator $\mathcal{L}^N $ is a symmetric, non-positive linear functional on $\ell^2 (\mathbb{T}^d_N ) $ when $u^N$ is regarded as a given function. Therefore, there exists a fundamental solution $p^N_2 (t, x, y ) $ corresponding to the operator $ \mathcal{L}^N_2 $ satisfying an derivative estimate \eqref{fundamental}. By virtue of Duhamel's principle applied to the first equation of the discrete reaction-diffusion system \eqref{dHDL eq}, we thus obtain 
\[
u^N (t, x ) = \sum_{ y \in \mathbb{T}^d_N } p^N_2 (t,x, y ) u^N (0, y ) - K \int_0^t \sum_{ y \in \mathbb{T}^d_N } p^N_2 (t-s, x, y ) u^N (s,y ) v^N (s, y) ds . 
\] 
Then, take the discrete derivative to $ j $-th direction and use Lemma \ref{heat ker} to get obtain 
\[
| \nabla^N u^N (t, x ) | \le C t^{-1/2 } + C K t^{ 1/2 }
\]
for some positive constant $C$, which holds true for every $t \in ( 0, T] $. In particular, we have $| \nabla^N u^N (t, x ) | \le C K $ for every $t \ge t_* $ and thus we complete the proof by combining with the above estimate which holds for small time $t \le t_*$.  
\end{proof}

Moreover, we have an $L^\infty$-estimate for second order discrete derivatives in the next Lemma \ref{ggrad u}. This is needed to estimate $\Delta^N \varphi (u^N (t, x ) ) $ when applying the Boltzmann-Gibbs principle for a possibly non-linear term which comes from the zero-range generator. The proof of Lemma \ref{ggrad u} can be done analogously in \cite{EFHPS20}. Since it is complicated and the main idea is similar to that of Lemma \ref{grad u}, we omit the proof here.

\begin{lemma}
\label{ggrad u}
There exists a positive constant $C$ such that 
\[
|  \nabla^N {}^t \! ( \nabla^N u^N (t, x ) ) | \le C K^3 e^{CK^2 t }
\]
for every $t \in [0, T ] $ and $x \in \mathbb{T}^d_N$ provided the assertion holds at time $t = 0 $. 
In particular, there exists a positive constant $C = C (T) $ such that  
\[
| \Delta^N \varphi (u^N (t, x ) ) | \le C K^3 e^{C K^2 }
\]
for every $t \in [0, T] $ and $x \in \mathbb{T}^d_N $. 
\end{lemma}

In the proof of the $L^\infty$-estimate for first order derivatives given in Lemma \ref{grad u}, we used Duhamel's principle twice. This is because the fundamental solution $p^N (t, x, y) $ may not has the symmetry in spatial variables $x $ and $y$ in general. Indeed, since the heat kernel generated by usual Laplacian is symmetric in spatial variables, we can deduce an $L^\infty $-estimate by applying Duhamel's principle directly in the second equation of \eqref{dHDL eq} for $v^N$ as follows.

\begin{lemma}
\label{grad v}
There exists a positive constant $C $ such that 
\[
| \nabla^N v^N (t, x ) | \le \varepsilon^{-1/2 } K (C_0 + C t^{1/2}  ) 
\]
for every $t \in [0, T ] $ and $x \in \mathbb{T}^d_N $ provided the assertion holds at time $t = 0 $. 
\end{lemma}
\begin{proof}
Let $p^N_0 (t, x, y ) $ be the heat kernel corresponding to the discrete Laplacian $\Delta^N $ on $\mathbb{T}^d_N $ satisfying the derivative estimate given in Proposition \ref{heat ker}. By applying Duhamel's principle to the second equation of \eqref{dHDL eq}, we have 
\[
v^N(t, x ) 
= \sum_{ y \in \mathbb{T}^d_N } v^N (0, y ) p^N_0 (\varepsilon t ,x, y ) 
- K \int_0^t \sum_{y \in \mathbb{T}^d_N } u^N (s, y ) v^N (s, y ) p^N_0 ( \varepsilon (t- s ) , x, y ) ds 
\] 
for every $t \in [0, T ] $ and $ x \in \mathbb{T}^d_N $. Then, taking the discrete gradient for both sides to conclude that 
\[
| \nabla^N  v^N (t, x) |  
\le \sum_{ y \in \mathbb{T}^d_N } | \nabla^N v^N (0, y) | p^N_0 (\varepsilon t, x, y ) 
+ C K \int_0^y \sum_{ y \in \mathbb{T}^d_N } | \nabla^N p^N_0 (\varepsilon (t- s ), x , y )|  ds .
\] 
Here we used the symmetry of the discrete heat kernel $p^N_0$ in spatial arguments to estimate the first term, while the second estimate can be deduced by the uniform boundedness of $u^N $ and $v^N $ (Lemma \ref{unif est}). Hence we complete the proof in view of the assumption and the derivative estimate \eqref{fundamental}. 
\end{proof}

\section{The multi-variable Boltzmann-Gibbs principle}
\label{sec:BG}
In this section, we prove the multi-variable Boltzmann-Gibbs principle (Theorem \ref{BG}), which is a crucial estimate to replace a multi-variable local function by a linear combination of rescaled variables $\omega_1 $ and $\omega_2$. First we give the proof of the Boltzmann-Gibbs principle with some main ingredients at hand which are to be proved in forthcoming subsections. Several estimates for truncation (Lemmas \ref{truncation1} and \ref{truncation2}) and main estimates for residual terms (Lemmas \ref{R1}, \ref{R2} and \ref{R3}) to prove the Boltzmann-Gibbs principle are given in Subsection \ref{subsec:truncation} and \ref{subsec:main}, respectively. In the sequel, we fix a local function $h$ with support in a finite square box $\Lambda_h \Subset \mathbb{T}^d_N$ satisfying the bound \eqref{BG assumption}. Recall here that for this local function $h$ we defined its residual term $R (x) = R (x, \eta) $ by 
\[
\begin{aligned}
R ( x ) = 
& \tau_x h ( \eta_1 , \eta_2 ) - \tilde{h} ( u^N (t, x ), v^N (t, x ) )  \\
&  - \partial_1 \tilde{h} (u^N (t, x), v^N (t, x) ) ( \eta_1 (x) - u^N (t, x) ) \\
& - \partial_2 \tilde{h} (u^N (t, x), v^N (t, x) ) ( \eta_2 (x) - v^N (t, x) ) . 
\end{aligned}
\]
Throughout this section, we denote by $\eta_i^\ell (x) = | \Lambda_\ell |^{-1} \sum_{  z \in \Lambda_\ell } \eta_i (x + z) $ local averages in a box with width $\ell \in \mathbb{N} $ for each $i = 1,2$ and $x \in \mathbb{T}^d_N $. Moreover, let  
\[
 y_1 (x) \coloneqq \frac{1}{ |\Lambda_\ell | } \sum_{ z \in \Lambda_{ \ell }  } ( \eta_1 (x + z) - u^N (t, x ) ) ,\quad 
 y_2 (x) \coloneqq \frac{1}{ |\Lambda_\ell | } \sum_{ z \in \Lambda_{ \ell }  } ( \eta_1 (x + z) - u^N (t, x ) )
\]
and let 
\[
\begin{aligned}
 \tilde{y}_1 (x) \coloneqq \frac{1}{ |\Lambda_\ell | } \sum_{ z \in \Lambda_{ \ell }  } ( \eta_1 (x + z) - u^N (t, x + z) ) , \quad
 \tilde{y}_2 (x) \coloneqq  \frac{1}{ |  \Lambda_\ell | } \sum_{ z \in \Lambda_{ \ell }  } ( \eta_2 (x + z) - v^N (t, x + z) ) 
\end{aligned}
\]
be scaled local averages around a site $x \in \mathbb{T}^d_N $. We write $y (x) = (y_1 (x) , y_2 (x ) ) $ and $\tilde{y} ( x ) = ( \tilde{y}_1 (x), \tilde{y}_2 (x) ) $, and we define their norms by $| y (x) | = | y_1 (x ) | + | y_2 (x) | $ and $ | \tilde{y} (x) | = | \tilde{y}_1 (x) | + | \tilde{y}_2 (x) | $, for every $x \in \mathbb{T}^d_N$. We first give the proof of our Boltzmann-Gibbs principle.

\begin{proof}[Proof of Theorem \ref{BG}]
First we conduct a truncation procedure which enables us to focus on bounded type-$1$ configuration where there are finite numbers of particles on each site. By Lemma \ref{truncation1} and \ref{truncation2}, we have
\[
\begin{aligned}
E_{ \mu^N_t } \bigg[ \bigg| \sum_{ x \in \mathbb{T}^d_N}  R (x) \bigg| \bigg] 
& \le E_{ \mu^N_t } \bigg[ \bigg| \sum_{ x \in \mathbb{T}^d_N } R (x)   \mathbf{1}_{ \{ \sum_{ z \in \Lambda_h } \eta_1 (x + z ) \le A \} } \bigg|  \bigg] 
+ C N^d   A^{-1}  \\
& \le E_{ \mu^N_t } \bigg[\bigg| \sum_{ x \in \mathbb{T}^d_N }  R (x)   \mathbf{1}_{ \{ \sum_{ z \in \Lambda_h }\eta_1 ( x + z ) \le A \} } \mathbf{1}_{ \{ \eta^\ell_1 (x)  \le B \} }  \bigg| \bigg] 
+ \frac{ C N^d }{  A } + \frac{ C A^2 N^d}{ B } 
\end{aligned}
\]
Now we decompose the summand in the last quantity as $R_1+ R_2 + R_3 $ where 
\[
\begin{aligned}
& R_1 (x) = \big(  R (x) \mathbf{1}_{ \{  \sum_{ z \in \Lambda_h } \eta_1 (x + z) \le A \} } 
 - E_{\nu_\beta } [ R (x) \mathbf{1}_{ \{  \sum_{z \in \Lambda_h } \eta_1 (x + z )  \le A \} } | \eta_1^\ell (x) , \eta_2^\ell (x) ] \big) 
\mathbf{1}_{ \{ \eta_1^\ell (x)  \le B \} } , \\
& R_2 (x) =  E_{\nu_\beta} [ R (x) \mathbf{1}_{ \{  \sum_{ z \in \Lambda_h } \eta_1 (x + z ) \le A \} } | \eta_1^\ell (x) , \eta_2^\ell (x) ] 
\mathbf{1}_{ \{ \eta_1^\ell (x) \le B \} } \mathbf{1}_{ \{ | y (x) | \le \delta \} } , \\
& R_3 (x) =  E_{\nu_\beta } [ R (x) \mathbf{1}_{ \{ \sum_{ z \in \Lambda_h}  \eta_1 (x + z ) \le A \} } | \eta_1^\ell (x) , \eta_2^\ell (x) ] 
\mathbf{1}_{ \{ \eta_1^\ell (x) \le B \} } \mathbf{1}_{ \{ | y (x) | > \delta \} }  
\end{aligned}
\]
with a sufficiently small positive constant $\delta $. In forthcoming subsections, we give quantitative estimates for $R_1$, $R_2$ and $R_3$ in Lemma \ref{R1}, \ref{R2} and \ref{R3}, respectively. Combining these estimates altogether, there exists a positive constant $C = C (\delta ) $ such that 
\[
\begin{aligned}
 \mathbb{E}^N \bigg[ \bigg| \int_0^T \sum_{x \in \mathbb{T}^d_N } R (x) dt  \bigg| \bigg]  
  \le & \int_0^T C H (\mu^N_t | \nu^N_t ) dt 
+  C A^{-1  }N^d 
+  C A B^{-1 }  N^d 
+  C \gamma_1^{-1 } KN^d \\
& +  C (N/ \ell )^d  e^{- c \ell^d} 
+  C (N/ \ell )^d   
+  C \ell^d N^{d-2} \varepsilon^{-1}  \gamma_1 \ell^2 A^2  \\
& + C \ell^2 N^{d-2} \varepsilon^{-1} ( B  + 1 )  
+ C N^d A^{-1} e^{C_1 K }
\end{aligned}
\]
where $A = N^{ \alpha_A }$, $B = N^{\alpha_B } $, $\gamma_1 = N^{ \alpha_{\gamma_1} }$, $\ell = N^{ \alpha_\ell } $ with positive integers $\alpha_A$, $\alpha_B $, $\alpha_{\gamma_1 } $ and $\alpha_\ell $ satisfying $\alpha_\varepsilon +  \alpha_A + \alpha_{\gamma_1 } + ( d + 2) \alpha_\ell - 2 < 0 $. Taking $B = A^2$ for convenience, the above display can further be bounded by 
\[
\int_0^T C H (\mu^N_t | \nu^N_t ) dt 
+ C A^{-1 } N^d 
 +  C \gamma_1^{-1} K N^d 
 +  C (N / \ell )^d  
+ C \varepsilon^{-1}  \ell^{d+2 } N^{d-2}  \gamma_1 A^2 + C A^{-1} e^{ C_1 K } 
\]
for sufficiently large $N$. Now we fix indices by $\varepsilon_0 = \alpha_\varepsilon = \alpha_A = \alpha_{\gamma_1 } = d \alpha_\ell = 2 - (\alpha_\varepsilon + 2 \alpha_A + (d+2 ) \alpha_\ell + \alpha_{ \gamma_1 }  )$ so that we may take $\varepsilon_0  = d / (3 d + 1 ) $. Therefore, we obtain
\[
\mathbb{E}^N \bigg[ \bigg| \int_0^T \sum_{x \in \mathbb{T}^d_N } R (x) dt  \bigg| \bigg] 
\le \int_0^T C H (\mu^N_t | \nu^N_t ) dt + O (K N^{d - \varepsilon_0 } )
\]
and thus we complete the proof of the Boltzmann-Gibbs principle (Theorem \ref{BG}). 
\end{proof}

\subsection{Preliminary estimates}
\begin{lemma}
\label{sum est}
Assume $H (\mu^N_0 | \nu^N_0 ) = O (N^d ) $ as $N$ tends to infinity. Then we have
\[
E_{\mu^N_t} \big[ \sum_{ x  \in \mathbb{T}^d_N } (  \eta_1 (x) + \eta_2 (x)  ) \big] = O (N ^d) .
\]
\end{lemma}
\begin{proof}
Since creation of new particles does not occur in our dynamics, applying the entropy inequality, we obtain for each $i = 1, 2$ that 
\[
\begin{aligned}
 E_{\mu^N_t} \big[ \sum_{ x  \in \mathbb{T}^d_N } \eta_i (x) \big] 
 \le E_{\mu^N_0 } \big[ \sum_{ x  \in \mathbb{T}^d_N } \eta_i (x)  \big] 
& \le \gamma^{-1 } H (\mu^N_0 | \nu^N_0 )   +\gamma^{-1 } \log E_{ \nu^N_0 } \big[ e^{ \gamma \sum_{x } \eta_i (x)  } \big]  \\
& \le \gamma^{-1}  H (\mu^N_0 | \nu^N_0 ) + \gamma^{-1} N^d  \max_{ x \in \mathbb{T}^d_N }  \log E_{ \nu^N_0 } \big[ e^{ \gamma \eta_i (x)  } \big] 
\end{aligned}
\]
for every positive $ \gamma $. However, since both zero-range processes have uniform exponential moment with sufficiently small $\gamma $ with respect to equilibrium measures, the second term in the right-hand side of the above display has order $O (N^d)$. Thus we complete the proof recalling the assumption $H (\mu^N_0 | \nu^N_0 ) = O (N^d ) $. 
\end{proof}

\begin{lemma}
\label{ent equi}
For every $\gamma > 0$, $\beta = (\beta_1 , \beta_2 ) \in (0, \infty ) \times ( 0, 1 )$ and $t \ge 0$, we have  
\[
H ( \mu^N_t | \nu_\beta) = H ( \mu^N_t | \nu^N_t ) + O (  N^d ) 
\]
as $N$ tends to infinity. In particular, $H ( \mu^N_ 0 | \nu_\beta ) = O ( N^d ) $ if $H (\mu^N_0 | \nu^N_0)  = O (N^d ) $. 
\end{lemma}
\begin{proof}
By definition of relative entropy, we have 
\[
H (\mu^N_t | \nu_\beta) = H (\mu^N_t | \nu^N_t ) + \int_{\mathcal{X}_N } \log \frac{ d \nu^N_t  }{ d\nu_\beta } d \mu^N_t . 
\]
Moreover, noting uniform boundedness of $u^N$ and $v^N$ by Lemma \ref{unif est} and that the partition function $Z_{ \rho_1 } = \sum_{ k \in \mathbb{Z}_+ } \varphi (\rho_1 )^k / g (k) ! $ and $\varphi ( \rho_1) $ are increasing in $\rho_1 \ge 0$, we have for every $k \in \mathbb{Z}_+$ 
\[
\frac{ d \nu^N_{1, t}  }{ d\nu_{ \beta_1} } ( \eta_1 (x) = k ) = \frac{ Z_{\beta_1 } }{ Z_{ u^N(t, x ) } } \frac{ \varphi (u^N (t, x ) )^k  }{ \varphi (\beta_1)^k } \le \frac{ Z_{\beta_1} }{ Z_{ e^{- C_1 K } } } \frac{ \varphi (M_u ) ^k  }{ \varphi (\beta_1)^k } .
\]
However, our zero-range jump rate $g $ satisfies $g ( k ) \ge C k $ by the assumption \textbf{(LG)}, we have $ Z_{ \rho_1 } \ge \sum_{ k \in \mathbb{Z}_+ } ( \varphi (\rho_1 ) / C )^k / k! = e^{ \varphi (\rho_1 )/ c  } \ge 1 $ for every $\rho_1 \ge 0 $. Therefore, we obtain 
\[
\log \frac{ d \nu^N_{1, t } }{ d \nu_{ \beta_1} }
 \le N^d \log Z_{ \beta_1 }  + \log \frac{ \varphi (M_u )}{ \varphi (\beta_1) } E_{ \mu^N_t } \big[ \sum_{ x \in \mathbb{T}^d_N }  \eta_1 (x)  \big] = O ( N^d )
\]
as $N$ tends to infinity, noting $E_{\mu^N_t } \big[ \sum_x \eta_1 (x) \big] = O (N^d) $ by Lemma \ref{sum est}. For type-$2$ configuration, on the other hand, we have a uniform bound
\[
\frac{ d \nu^N_{2, t}  }{ d\nu_{ \beta_2} } (\eta_2(x) ) = \frac{ v^N (t, x ) }{ \beta_2 } \eta_2 (x) +  \frac{1- v^N (t, x) }{ 1- \beta_2 } (1- \eta_2 (x)) \le C 
\]
so that $E_{\mu^N_t } [ \log d\nu^N_{2, t} / d \nu_{\beta_2 } ]$ has order $O (N^d)$ and thus we complete the proof. 
\end{proof}

\subsection{Spectral gap estimates}
Next we give some preliminary results for spectral gaps. Before that, we introduce the restriction of zero-range and Kawasaki processes on single component as follows. Let $L^o_Z$ and $L^o_K $ be generators corresponding to zero-range and Kawasaki dynamics with one species, respectively. Namely, for any $f : \mathcal{X}^1_N \to \mathbb{R} $ the zero-range generator $L^o_Z$ acts as 
\[
L^o_Z f (\eta_1) = \sum_{x, y \in \mathbb{T}^d_N , |x-y| = 1 } g (\eta_1 (x)) [ f (\eta_1^{x,y} ) - f ( \eta_1)  ]  ,
\] 
while the Kawasaki generator $L^o_K$ acts for any $f : \mathcal{X}^2_N \to \mathbb{R} $ as 
\[
L^o_K f (\eta_2) = \sum_{ x, y \in \mathbb{T}^d_N , |x -y |= 1 } \eta_2 (x) ( 1 - \eta_2 (y ) ) [ f  (\eta_2^{x,y } ) - f  (\eta_2)] . 
\]
For each $\ell \in \mathbb{Z}_+ $, let $\Lambda_\ell = [-\ell,  \ell ]^d \cap \mathbb{Z}^d $ and let $L^o_{Z,\ell } $ and $L^o_{K , \ell } $ be restriction of $L^o_Z $ and $L^o_K$ on $\Lambda_\ell$, respectively, defined by
\[
L^o_{Z, \ell } f (\eta_1 ) = \sum_{ x, y \in \Lambda_\ell , |x -y | = 1 } g (\eta_1 (x) ) [ f ( \eta_1^{x, y} ) - f (\eta_1 )]
\] 
and
\[
L^o_{K, \ell } f (\eta_2 ) = \sum_{ x, y \in \Lambda_\ell , |x -y | = 1 } \eta_2 (x) ( 1 - \eta_2 (y) )  [ f ( \eta_2^{x, y} ) - f (\eta_2 )] .
\] 
Given a weight $u = \{ u ( x ) : x \in \Lambda_\ell \} \in \mathbb{R}_+^{\Lambda_\ell } $, we denote the restriction on $\Lambda_\ell $ of the product measure $\nu_{ 1, u  } $ by $\nu_{1, u , \ell } $. When the weight is spatially homogeneous, i.e. $u (x) = \beta_1 $ for every $x \in \Lambda_\ell $, we write $\nu_{1 , u , \ell } = \nu_{ 1, \beta_1 , \ell } $ with abuse of notation. When there are $j_1 $ ($j_1 \in \mathbb{Z}_+ $) type-$1$ particles on $\Lambda_\ell $, the canonical measure $\nu_{1, \beta_1 , \ell, j_1 } = \nu_{ 1, \beta_1 , \ell } (\,  \cdot \,  | \sum_{ x \in \Lambda_\ell } \eta_1 (x) = j_1 ) $, which is in fact independent of $\beta_1$, is invariant for a process generated by $L^o_{Z, \ell } $. Similarly, for every $j_2 \in \mathbb{Z}_+$ and $\beta_2 \in [0, 1 ] $, a measure $\nu_{2, \beta_2 , \ell , j_2 } =  \nu_{ 2, \beta_2 , \ell } (\,  \cdot \, | \sum_{ x \in \Lambda_\ell } \eta_2 (x)  = j_2 ) $ is defined and it turns out to be invariant under the dynamics generated by $L^o_{K, \ell} $. Hereafter we denote by $L^o_{Z, \ell, j_1 } $ (resp. $L^o_{K, \ell, j_2 } $) the zero-range (resp. Kawasaki) generator acting on every real-valued functions on type-$1$ (resp. type-$2$) configurations such that the total numbers of type-$1$ (resp. type-$2$) particles is $j_1 $ (resp. $j_2$). 


\begin{proposition}[\cite{LSV96}]
\label{ZRP spec}
Assume the jump rate $g $ satisfies \textbf{(LG)}. Then there exists a positive constant $C  $ such that the zero-range generator $L^o_{Z, \ell , j } $ satisfies a spectral gap estimate with intensity $\gamma_1^{- 1 } \ge C \ell^{ -2 }  $ for every $\ell , j \in \mathbb{Z}_+ $, that is, $
\mathrm{Var}_{ \nu_{ 1 , \beta_1 , \ell, j } } [f ] \le \gamma_1 \langle f, - L^o_{ Z, \ell ,j } f \rangle_{ L^2 (\nu_{1, \beta_1 , \ell, j } ) }  $ for every real-valued function $f $ on the configuration space of the zero-range generator $L^o_{ Z , \ell , j } $. 
\end{proposition}


\begin{proposition}[\cite{LY93}]
\label{Kawasaki spec}
There exists a positive constant $C  $ such that the Kawasaki generator $L^o_{ K, \ell, j }$ satisfies a spectral gap estimate with intensity $\gamma_2^{- 1 } \ge C \ell^{- 2 } $ for every $\ell , j \in \mathbb{Z}_+$, that is, $ \mathrm{Var}_{ \nu_{ 2 , \beta_2 , \ell, j } } [f ] \le \gamma_2 \langle f, - L^o_{K, \ell ,j } f \rangle_{ L^2 (\nu_{2, \beta_2 , \ell, j } ) } $ for every real-valued function $f $ on the configuration space of the Kawasaki generator $L^o_{ K , \ell , j } $. 
\end{proposition}

Note that when the small parameter $\varepsilon (N)$ tends to zero, the spectral gap for our Kawasaki dynamics diverges. Since our diffusion dynamics generated by $N^2 L_Z + \varepsilon (N) L_K $ is governed by both zero-range dynamics for type-$1$ particles and Kawasaki dynamics for type-$2$ particles, we get a spectral gap estimate by combining the above result for Kawasaki and zero-range dynamics. 

\begin{lemma}
\label{total spec}
We fix any $\ell \in \mathbb{Z}_+$ and $(j_1 , j_2 ) \in \mathbb{Z}_+^2$. Let $\gamma_1 $ be a spectral gap for the zero-range generator $L^o_{Z, \ell, j_1 }$ and let $\gamma_2 $ be a spectral gap for the Kawasaki generator $L^o_{K, \ell, j_2 } $. Then, the generator $L_{Z, \ell, j_1 } + \varepsilon L_{ K, \ell, j_2 } $ satisfies a spectral gap estimate with intensity $ ( 2 \gamma_1 + 2 \varepsilon^{-1} \gamma_2 ) ^{-1} $. 
\end{lemma}
\begin{proof}
We take an arbitrary function $f : \mathcal{X}_N \to \mathbb{R} $. We abbreviate the parameter $\beta = (\beta_1 , \beta_2 ) $ and $\ell \in \mathbb{N} $ here, and recall that our reference measure has a product form $\nu = \nu_1 \otimes \nu_2$ with this convention. By an elementary inequality $(a + b )^2 \le 2 a^2 + 2 b^2 $ for $a, b \in \mathbb{R}$, we have
\[
\text{Var}_{ \nu_{1 , j_1} \otimes \nu_{2, j_2 }  } [ f ] \le 2 E_{ \nu_{1, j_1 } \otimes \nu_{2 ,j_2} } [ (f - E_{ \nu_{1, j_1 } } [f] )^2 ]  + 2 E_{\nu_{1, j_1 } \otimes \nu_{2 ,j_2} } [ ( E_{ \nu_{1, j_1  } } [ f ] - E_{ \nu_{1, j_1 } \otimes \nu_{2 , j_2} } [f] )^2 ] . 
\]
Since the zero-range generator satisfies a spectral gap estimate with intensity $\gamma_1^{-1}$, the first term in the above display can be bounded above by 
\[
 2 \gamma_1 E_{\nu_{2, j_2 } } [ \langle f , - L_{Z, \ell, j_1 } f \rangle_{ L^2 (\nu_{1, j_1 }  ) } ] = 2 \gamma_1 \langle f , - L_{Z, \ell, j_1 } f \rangle_{ L^2 (\nu_{1, j_1 } \otimes \nu_{ 2, j_2 } ) } . 
\]
On the other hand, due to a spectral gap estimate for Kawasaki dynamics, the second term can be bounded above by 
\[
2 \varepsilon^{-1} \gamma_2  \langle E_{\nu_{1. j_1}  } [f] , - \varepsilon L_{K, \ell, j_2 } E_{\nu_{1, j_1} } [f] \rangle_{ L^2 (\nu_{2, j_2}  ) }
\le 2 \varepsilon^{-1} \gamma_2  \langle f , - \varepsilon L_{K, \ell, j_2 } f \rangle_{ L^2 (\nu_{1, j_1 } \otimes \nu_{2, j_2 } ) } .
\] 
Here in the last estimate we used Jensen's inequality noting the Dirichlet energy for Kawasaki dynamics $f \mapsto \langle f,  - L_{ K, \ell, j_1 } f \rangle_{ L^2 ( \nu_2 ) } $ is a convex functional which can be deduced from convexity of quadratic functions. Hence we obtain the desired spectral gap estimate for total generator $ L_{ Z, \ell, j_1 } + \varepsilon L_{K, \ell, j_2  } $ and complete the proof.  
\end{proof}

We write $L_{Z, \ell, x , j_1 } \coloneqq L_{Z, \ell , x , j_1 } \tau_x$ and $L_{K, \ell, x , j_1 } \coloneqq L_{K, \ell , x , j_1 } \tau_x$ where $\{ \tau_x \}_{ x \in \mathbb{Z}^d } $ is a shift acting for any real-valued function $f$ on $\mathcal{X}_N$ as $\tau_x f (\cdot ) = f (\tau_x \cdot ) $ where $\tau_x \eta (z) = \eta (x + z )$ for every $\eta \in \mathcal{X}_N $. Note here that spectral gap is not affected by spatial shifts since our dynamics is translation-invariant. Hence for any $\ell \in \mathbb{Z}_+ $, $x \in \mathbb{T}^d_N $ and $(j_1 , j_2 ) \in \mathbb{Z}_+^2$, due to the above lemmas, we see that our diffusion dynamics generated by $L_{Z, \ell, x, j_1 } + \varepsilon L_{K, \ell, x, j_2 } $ has a spectral gap with intensity 
\[ 
( 2 \gamma_1 + 2 \varepsilon^{-1} \gamma_2 )^{ -1 } \ge C \big( \ell^{2}  + \varepsilon^{-1 } \ell^{2}  \big)^{-1} \ge C \varepsilon \ell^{ - 2 }  
\]
for sufficiently small $\varepsilon = \varepsilon (N) $.

\subsection{Truncation estimates}
\label{subsec:truncation}

\begin{lemma}
\label{truncation1}
Let $A = A_N = N^{ \alpha_A }$ where $\alpha_A$ is a positive constant. Then there exists a positive constant $C$ such that 
\[
E_{ \mu^N_t } \bigg[ \sum_{ x \in \mathbb{T}^d_N } | R (x) | \mathbf{1}_{ \{ \sum_{ z \in \Lambda_h } \eta_1 (x + z )  > A \} } \bigg] 
\le C N^d A^{-1} 
\]
for sufficiently large $N$. 
\end{lemma}
\begin{proof}
By the entropy inequality with any positive constant $\gamma $, we can bound the left-hand side of the assertion as 
\[
\begin{aligned}
& E_{ \mu^N_t } \big[ \sum_x | R (x) | \mathbf{1}_{ \{ \sum_{ y \in \Lambda_h } \eta_1 ( x + y) > A   \}  } \big]  \\
& \quad \le \gamma^{-1} H (\mu^N_t | \nu^N_t ) +  \gamma^{-1} \log  E_{ \nu^N_t } \big[  e^{ \gamma \sum_x | R (x) | \mathbf{1}_{ \{ \sum_{ y \in \Lambda_h } \eta_1 (x + y ) > A  \} } } \big] \\
& \quad \le \gamma^{-1} H (\mu^N_t | \nu^N_t ) +  \gamma^{-1} | \Lambda_h |^{-1 } \sum_{  x } \log  E_{ \nu^N_t } \big[  e^{ \gamma | \Lambda_h |  | R (x) | \mathbf{1}_{ \{ \sum_{ y \in \Lambda_h } \eta_1 (x + y ) > A  \} } } \big] \\
&\quad = \gamma^{-1} H (\mu^N_t | \nu^N_t ) \\
& \qquad + \gamma^{-1} | \Lambda_h |^{-1}  \sum_{ x } \log \bigg( 1- \nu^N_t \big( \sum_{ y \in \Lambda_h } \eta_1 (x + y ) > A  \big) + E_{ \nu^N_t } \big[ e^{ \gamma | \Lambda_h | | R (x) | } \mathbf{1}_{ \{ \sum_{ y \in \Lambda_h } \eta_1 (x + y ) > A \} }  \big] \bigg) 
\end{aligned}
\]
where in the second estimate we decomposed $\mathbb{T}^d_N $ into $| \Lambda_h | $ regular sublattices \\
$\{ \tau_z ( | \Lambda_h |^{1/d } / N \mathbb{Z} )^d  ; z \in \Lambda_h \}$ and applied H\"{o}lder's inequality. Furthermore, using an elementary inequality $\log (1 + x ) \le x $ for $x \ge 0$ and then Markov's inequality, the second term in the last display can be bounded from above by 
\[
\gamma^{-1} | \Lambda_h | A^{-1} \sum_{x \in \mathbb{T}^d_N } E_{\nu^N_t } \big[ \sum_{ y \in \Lambda_h } \eta_1 (x + y  ) e^{\gamma | \Lambda_h | \, | R (x) | }   \big] . 
\]
However, since $ \sup_{ x } E_{ \nu^N_t } \big[ \sum_{ y \in \Lambda_h } \eta_1 (x + y ) e^{\gamma  |\Lambda_h | \, | R (x) | }\big] $ stays finite for sufficiently small $\gamma$ which is independent of $N$, the last quantity has order $ O (N^d A^{-1 } ) $ and thus we complete the proof.    
\end{proof}

\begin{lemma}
\label{truncation2}
Let $B = B_N =  N^{ \alpha_B } $ with a positive constant $\alpha_B $ and let $\ell $ be a positive integer. Then there exists a positive constant $C$ such that 
\[
E_{ \mu^N_t } \bigg[ \sum_{ x \in \mathbb{T}^d_N }| R (x) | \mathbf{1}_{ \{ \sum_{ z \in \Lambda_h } \eta_1 (x + z ) \le A \} } \mathbf{1}_{ \{ \eta^\ell_1 (x)  > B \} }  \bigg] 
\le  C  N^d  A B^{-1}  
\]
for sufficiently large $N $. 
\end{lemma}
\begin{proof}
Since the local function $h$ satisfies the bound \eqref{BG assumption}, we see that $| R (x) | \mathbf{1}_{ \{ \sum_{ y \in \Lambda_h } \eta_1 (x + y ) \le A \} } $ has order $O (A) $ as $N $ tends to infinity. Therefore, applying Markov's inequality to the indicator function $\mathbf{1}_{ \{ \eta_1^\ell (x) > B \} } $, the left-hand side of the assertion can be bounded from above by 
\[
C A B^{-1}  E_{ \mu^N_t } \big[ \sum_{ x \in \mathbb{T}^d_N } \eta_1^\ell (x) \big]
\]
with some positive constant $C$. However, noting an identity $\sum_x \eta_1^\ell (x) = \sum_x \eta_1 (x) $, Lemma \ref{sum est} for the total number of particles finishes the proof. 
\end{proof}

\subsection{Main estimates}
\label{subsec:main}

\subsubsection{Estimate of $R_1$}
First we give an quantitative estimate for $R_1$ by using a spectral gap estimate.

\begin{lemma}
\label{R1}
Assume $H (\mu^N_0 | \nu^N_0 ) = O (N^d )$. Let $\ell  = N^{ \alpha_\ell }$ and $\gamma_1 =N^{ \gamma_1 }$ for positive constants $\alpha_\ell $ and $\alpha_{ \gamma_1 } $. Suppose $\alpha_\varepsilon +  \alpha_A + \alpha_{\gamma_1 } + ( d + 2) \alpha_\ell - 2 < 0 $ holds true. Then we have
\[
\mathbb{E}^N \bigg[ \bigg| \int_0^T  \sum_{ x \in \mathbb{T}^d_N } R_1 (x) dt  \bigg| \bigg]
 \le C \gamma_1^{-1} K N^d 
 + C \varepsilon^{-1} \gamma_1 \ell^{d+2} A^2 N^{d - 2} .
\]
\end{lemma}
\begin{proof}
Applying the entropy inequality, we have that 
\[
\mathbb{E}^N \bigg[ \bigg|  \int_0^T \sum_{ x \in \mathbb{T}^d_N}  R_1 (x) dt \bigg|  \bigg]
\le \gamma_1^{-1} H ( \mu^N_0 | \nu_\beta )  + \gamma_1^{-1} \log E_{ \nu_\beta } [ e^{ \gamma_1 |  \int_0^T \sum_x R_1 (x) dt |  } ]  
\]
for any positive constant $\gamma_1 $. Since $H ( \mu^N_0 | \nu_\beta ) = O ( K N^d ) $ provided $H (\mu^N_0 | \nu^N_0 ) = O (N^d) $ by Lemma \ref{ent equi}, only the second term is concerned. According to an elementary estimate $e^{|z|} \le e^z + e^{ -z }$ for every $z \in \mathbb{R}$ and the Feynman-Kac formula, the second term in the last display is bounded from above by
\begin{equation}
\label{R1sup1}
2 \int_0^T \sup_{ h } \bigg\{ \big\langle \sum_{ x } R_1 (x) , h  \big\rangle_{ L^2 (\nu_\beta ) }
 -  \gamma_1^{-1}  \mathcal{D}_N ( \sqrt{h} )  \bigg\}  dt  
\end{equation}
where $h$ is a density with respect to $\nu_\beta$. Moreover, $\mathcal{D}_N (\sqrt{h} ) = \mathcal{D}_N ( \sqrt{h} ; \nu_\beta )$ is a Dirichlet energy of $\sqrt{ h } $ with respect to the probability measure $\nu_\beta $ and a symmetric quadratic form $\mathcal{D}_N $ is defined by 
\[
\mathcal{D}_N  (f ) \coloneqq \langle f, - L_N f \rangle_{ L^2 (\nu_\beta) } 
= \langle f, - S_N f \rangle_{ L^2 (\nu_\beta) } , \quad
S_N = ( L_N + L_N^{ * , \nu_\beta} ) / 2 
\]
for every real-valued functions $f$ on $\mathcal{X}_N$. Hereafter we fix an arbitrary non-negative function $f$ on $\mathcal{X}_N $ such that $f^2 $ is density with respect to $\nu_\beta$ to bound the supremum in \eqref{R1sup1}. Since our dynamics is a superposition of diffusion and reaction dynamics, $\mathcal{D}_N$ is decomposed into sum of three terms $ N^2 \mathcal{D}_Z + \varepsilon (N) N^2 \mathcal{D}_K + K ( N ) \mathcal{D}_G $ where 
\[
\begin{aligned}
& \mathcal{D}_Z ( f ) 
\coloneqq \langle f, - L_Z f \rangle_{ L^2 (\nu_\beta) } 
 = \frac{ 1 }{ 2 } \sum_{ x, y, \in \mathbb{T}^d_N ; | x -y | = 1 } E_{ \nu_\beta } \big[ g_1 ( \eta_x) ( f (\eta_1^{x, y } , \eta_2 ) - f (\eta_1 , \eta_2)  )^2  \big] , \\
& \mathcal{D}_K ( f ) 
\coloneqq \langle f, - L_Z f \rangle_{ L^2 (\nu_\beta) } 
 = \frac{ 1 }{ 2 } \sum_{ x, y, \in \mathbb{T}^d_N ; | x -y | = 1 } E_{ \nu_\beta } \big[  ( f (\eta_1 , \eta_2^{x, y } ) - f (\eta_1 , \eta_2)  )^2  \big]  , 
\end{aligned}
\]
and $\mathcal{D}_G ( f ) \coloneqq \langle f , - L_G f \rangle_{ L^2 (\nu_\beta ) } $ can be calculated as 
\[
\begin{aligned}
\mathcal{D}_G (f ) 
& = - \sum_{ x } E_{ \nu_\beta } \big[ \eta_1 (x) \eta_2 (x)  f (\eta_1 , \eta_2 ) (  f (\eta^{x, -}_1 , \eta_2^x ) - f (\eta_1 , \eta_2 ) )  \big] \\
& =  \sum_{ x } E_{ \nu_\beta } \big[ \eta_1 (x) \eta_2 (x)  (  f (\eta^{x, - }_1 , \eta_2^x ) - f (\eta_1 , \eta_2 ) )^2   \big]  \\
& \quad  + \sum_{ x } E_{ \nu_\beta } \big[ \eta_1 (x) \eta_2 (x) f (\eta_1, \eta_2 ) f ( \eta_1^{x, - } , \eta_2^x ) \big] \\
& \quad - \sum_{ x } E_{ \nu_\beta } \bigg[ f (\eta_1 ,\eta_2 )^2 ( \eta_1 (x) + 1 ) ( 1- \eta_2 (x) )  \frac{ \varphi ( \beta_1 ) }{  g ( \eta_1 (x) + 1 ) } \frac{ \beta_2 }{ 1 - \beta_2  }  \bigg] . 
\end{aligned}
\]
The first two terms in the last display are non-negative and the third term has order $O ( N^d ) $ according to the assumption \textbf{(LG)}, noting $f^2$ is density with respect to $\nu_\beta$. Therefore, the supremum in \eqref{R1sup1} is bounded above by 
\begin{equation}
\label{R1sup2}
\sup_{ h } \bigg\{ \big\langle \sum_{ x } R_1 (x) , h  \big\rangle_{ L^2 (\nu_\beta ) }
 -  \gamma_1^{-1} N^2 \big( \mathcal{D}_Z ( \sqrt{h} ) + \varepsilon \mathcal{D}_K (\sqrt{h} )  \big)  \bigg\} + C K N^d  \gamma_1^{ -1 }
\end{equation}
where again the supremum is taken over all densities $h$ with respect to an equilibrium measure $\nu_\beta$. 

To analyze further, let $L_{ Z, \ell, x }$ and $L_{K, \ell, x }$ be the zero-range and the Kawasaki generators restricted on a block $\Lambda_{ \ell, x} = x +  [ - \ell , \ell  ]^d \cap \mathbb{Z}^d $, which are defined by
\[
\begin{aligned}
& L_{ Z, \ell, x  } f (\eta_1 , \eta_2 ) = \sum_{  y, z  \in \Lambda_{ \ell, x } , | y - z | = 1 } g (\eta_1 ( y ) ) \big[ f ( \eta_1^{y, z }, \eta_2 ) - f (\eta_1, \eta_2 )  \big] , \\ 
& L_{ K, \ell, x  } f (\eta_1 , \eta_2 ) = \sum_{  y, z  \in \Lambda_{ \ell, x } , | y - z | = 1 } \eta_2 (y) (1 - \eta_2 (z ) ) \big[ f ( \eta_1, \eta_2^{y, z } ) - f (\eta_1, \eta_2 )  \big]  
\end{aligned}
\]
for every real-valued functions $f $ on $\mathcal{X}_N$. Then, the associated Dirichlet energy $\mathcal{D}_{ Z, \ell, x } (f ; \nu_\beta ) = \langle f, - L_{ Z, \ell, x } f \rangle_{ L^2 ( \nu_\beta ) }$ and $\mathcal{D}_{ K, \ell, x } (f ; \nu_\beta ) = \langle f, - L_{ K, \ell, x } f \rangle_{ L^2 ( \nu_\beta ) }$ become to be 
\[
\begin{aligned}
& \mathcal{D}_{ Z, \ell, x  } (f; \nu_\beta) 
= \frac{ 1 }{ 2 } \sum_{  y, z  \in \Lambda_{ \ell, x } , | y - z | = 1 } E_{ \nu_\beta } \big[ g (\eta_1 ( y ) ) \big( f ( \eta_1^{y, z }, \eta_2 ) - f (\eta_1, \eta_2 )  \big)^2 \big]   , \\
& \mathcal{D}_{ K , \ell, x  } (f; \nu_\beta) 
= \frac{ 1 }{ 2 } \sum_{  y, z  \in \Lambda_{ \ell, x } , | y - z | = 1 } E_{ \nu_\beta } \big[ \eta_2 (y) ( 1- \eta_2 (z) ) \big( f ( \eta_1, \eta_2^{y, z } ) - f (\eta_1, \eta_2 )  \big)^2 \big]  . 
\end{aligned}
\]
Here we denote by $\mathcal{D} \coloneqq \mathcal{D}_Z + \varepsilon \mathcal{D}_K $ the Dirichlet energy corresponding to our diffusion dynamics and we define $\mathcal{D}_{\ell, x } \coloneqq \mathcal{D}_{Z, \ell, x} + \varepsilon \mathcal{D}_{ K, \ell, x } $. Counting the overlaps, we observe 
\[
\sum_{ x \in \mathbb{T}^d_N } \mathcal{D}_{ \ell ,x } (f ; \nu_\beta) = (2 \ell + 1)^d \mathcal{D} (f ; \nu_\beta)  
\] 
for every $\ell \in \mathbb{N}$. Therefore, the supremum in \eqref{R1sup2} can be bounded from above by 
\[
\begin{aligned}
\sum_{ x \in \mathbb{T}^d_N } \sup_{ h } \bigg\{ \big\langle  R_1 (x) , h  \big\rangle_{ L^2 (\nu_\beta ) }
 -  \gamma_1^{-1} N^2 (2 \ell + 1 )^{-d} \mathcal{D}_{ \ell, x }  ( \sqrt{h} ; \nu_\beta)  \bigg\} 
\end{aligned}
\]
for every $\ell \in \mathbb{N}$. By conditioning on the number of particles on $\Lambda_{ \ell, x } $, and dividing and multiplying by $E_{ \nu_\beta } [h | \sum_{z \in \Lambda_{\ell ,x } } \eta_1(z ) = j_ 1 , \sum_{ z \in \Lambda_{ \ell , x } } \eta_2 (x) = j_2 ]  $, the supremum in the last quantity is further bounded from above by
\[
\begin{aligned}
\sup_{ j \le B (2 \ell + 1 )^d } \sup_{ h } \bigg\{ \big\langle  R_1 (x) , h  \big\rangle_{ L^2 (\nu_{ \ell, x, j_1, j_2 } ) }
 -  \gamma_1^{-1} N^2 (2 \ell + 1 )^{-d} \mathcal{D}_{ \ell, x , j_1, j_2 }  ( \sqrt{h} )  \bigg\}  .
\end{aligned}
\]
Here the supremum is taken over all densities $h$ with respect to $\nu_{ \ell ,x, j_1, j_2  } $. Furthermore, according to the Rayleigh estimate given in \cite{KL99} (Theorem 1.1 in Appendix 3), the last display is less than or equal to 
\begin{equation}
\label{R1sup3}
 \gamma_1 N^{ -2 } (2 \ell + 1 )^d  
\frac{ \langle R_1 (x) , ( - L_{ Z, \ell ,x  } - \varepsilon L_{K, \ell, x } )^{-1} R_1 ( x ) \rangle_{ \nu_{ \ell, x, j_1 , j_2  } } }{ 1 - 2 \| R_1 (x) \|_{ \infty  } \gamma_1 N^{-2} ( 2 \ell + 1 )^d  gap (\ell, j_1, j_2 )^{-1} } 
\end{equation}
for every $\ell \in \mathbb{N}$, $x \in \mathbb{T}^d_N$ and $j \in \mathbb{N}$ such that $j_1 \le B (2 \ell + 1 )^d $. Here, $gap (\ell, x ) = ( 2 \gamma_1 + 2 \varepsilon^{-1 } \gamma_2 )^{ - 1 }  $ denotes a spectral gap of $L_{Z, \ell, x } + \varepsilon L_{K, \ell, x } $ acting on any test functions on $\{ \eta \in \mathbb{Z}_+^{\Lambda_{\ell, x } } \times \{ 0, 1 \}^{ \Lambda_{\ell, x } } ;  \sum_{x \in \Lambda_{ \ell ,x }  } \eta_1 (x) = j_1, \sum_{ x \in \Lambda_{ \ell ,x } }  \eta_2 (x)  = j_2  \} $ (see Lemma \ref{total spec}). Moreover, $\mathscr{H}^{-1}$-norm in the numerator of \eqref{R1sup3} can be estimated as 
\[
\langle R_1 (x) , ( - L_{ Z, \ell ,x } - \varepsilon L_{K, \ell, x } )^{-1} R_1 ( x ) \rangle_{ \nu_{ \ell, x, j_1, j_2 } } 
\le gap(\ell , j_1, j_2  )^{ -1 } \| R_1 (x) \|^2_{ \infty} 
\le C gap (\ell ,j_1 , j_2  )^{-1} A^2 
\]
noting $\| R_1 (x) \|_{ \infty } = O (A ) $ by definition. As a consequence of Lemma \ref{total spec}, we have $gap (\ell ,j_1, j_2 )^{ - 1 } \le C \varepsilon^{-1} \ell^2 $. 
Now we choose indices in order that $\alpha_{ \varepsilon } + \alpha_A + \alpha_{\gamma_1 } + ( d + 2) \alpha_\ell - 2 < 0 $ holds. Then  
\[
\| R_1 (x) \|_{ \infty  } \gamma_1 ( 2 \ell + 1 )^d N^{-2} gap (\ell, j_1, j_2 )^{-1} \le C A \gamma_1 \ell^d gap(\ell, j_1, j_2 )^{-1} N^{ -2 } = o_N (1)
\]
as $N$ tends to infinity so that the denominator in \eqref{R1sup3} is bounded from below by a positive constant which is independent of $N$. Hence the display \eqref{R1sup3} can be bounded from above by $ C \varepsilon^{ - 1 } \gamma_1 N^{-2} \ell^{ d + 2 } A^2 $ and combining all the above estimates we complete the proof. 
\end{proof}

\subsubsection{Estimate of $R_2$}
We give an estimate for $R_2 = R_2 (x )$. 

\begin{lemma}
\label{R2lem1}
There exists a positive constant $C$ such that 
\[
\big| E_{  \nu_\beta } [ R (x) \mathbf{1}_{ \{ \sum_{ z \in \Lambda_h } \eta_1 (x + z ) \le A \}  } |  \eta_1^\ell (x) , \eta_2^\ell (x)   ] \big|  \mathbf{1}_{ \{ | y (x) |  \le \delta  \}  }
 \le C  | y (x)  |^2 \mathbf{1}_{ \{ |y (x) | \le \delta  \} } + C \ell^{-d} + C A^{-1} e^{ C_1 K }
\]
holds true for every $x \in \mathbb{T}^d_N$. 
\end{lemma}
\begin{proof}
With the help of the equivalence of ensembles (Corollary A.1.7. in \cite{KL99}), when $| y (x) | \le \delta$ we have  
\[
\begin{aligned}
& \big| E_{ \nu_\beta } [ R (x) \mathbf{1}_{ \{ \sum_{ y \in \Lambda_h } \eta_1 (x + y )  \le A \} } | \eta_1^\ell (x) , \eta_2^\ell (x) ] \big| \\
& \quad 
\le \big| E_{ \nu_{ (u^N ( t, x ) , v^N (t, x ) ) + y (x) } } [ R (x) \mathbf{1}_{ \{ \sum_{ y \in \Lambda_h } \eta_1 (x + y ) \le A \} } ] \big| + C \ell^{-d} . 
\end{aligned}
\] 
Recall here that the measure $\nu_{ (u^N (t, x ) , v^N (t, x ) ) + y (x) } = \nu_{ ( \eta_1^\ell (x) , \eta_2^\ell (x) ) } $ is a product measure with spatially homogeneous weight $ (\eta_1^\ell (x)  , \eta_2^\ell (x) ) $. We choose `chemical potentials' $\lambda_1 =\lambda_1 (y_1 (x) )$ and $\lambda_2 = \lambda_2 (y_2 ( x ) ) $ by 
\[
\begin{aligned}
& \frac{ E_{ \nu_{( u^N (t,x ) , v^N (t, x ) ) } } [ \eta_1 (x) e^{ \lambda_1 (\eta_1 (x) - u^N (t, x) ) } ] }{ E_{ \nu_{( u^N (t,x ) , v^N (t, x ) ) } } [  e^{ \lambda_1 (\eta_1 (x) - u^N (t, x) ) } ] } 
= u^N (t, x) + y_1 (x)  , \\
& \frac{ E_{ \nu_{( u^N (t,x ) , v^N (t, x ) ) } } [ \eta_2 (x) e^{ \lambda_2 (\eta_2 (x) - v^N (t, x) ) } ] }{ E_{ \nu_{( u^N (t,x ) , v^N (t, x ) ) } } [  e^{ \lambda_2 (\eta_2 (x) - v^N (t, x) ) } ] }
= v^N (t, x) + y_2 (x)  .
\end{aligned}
\]
Then for each $i = 1, 2$, one can notice $dy_i (x) / d \lambda_i \ge 0 $ so that there is a one-to-one correspondence between $y_i $ and $\lambda_i $. In particular, we have $y_i (x) |_{ \lambda_i = 0} = 0$ by definition, which deduces $\lambda_i (0) = 0 $ for each $i = 1,2 $. Moreover, we see that 
\[
\lambda_1^\prime (0)  = \frac{ d \lambda_1  }{ d y_1 (x)  } (0) = E_{ \nu_{( u^N (t,x ) , v^N (t, x ) ) } } [ ( \eta_1 (x) - u^N (t, x) )^2  ]^{-1} 
\]
and similarly for type-$2$ configuration. In particular, we have $\lambda_1^\prime ( 0 ) = \chi_1 (u^N (t, x ) )^{ - 1 }$ and $\lambda_2^\prime (0) = \chi_2 ( v^N (t, x ) )^{-1}  $ where $\chi_1 (\rho_1 ) = \varphi (\rho_1) / \varphi^\prime (\rho_1 )$ ($\rho_1 \ge 0 $) and $\chi_2 (\rho_2 ) = \rho_2 (1- \rho_2) $ ($0 \le \rho_2 \le1 $) are incompressibility for zero-range and exclusion processes, respectively. Let $b (x) = R (x) \mathbf{1}_{ \{  \sum_{ z \in \Lambda_h } \eta_1 (x + z ) \le A \} }$. Then, a direct computation shows  
\[
\begin{aligned}
& \frac{ \partial }{ \partial y_1 (x)  } E_{ \nu_{ (u^N (t, x ) , v^N (t, x) ) + y (x) } } [ b (x) ]  \bigg|_{ y (x) = 0 } 
=  E_{\nu^N_t }  [ b (x) ( \eta_1 (x) - u^N (t, x) ) ] / \chi_1 (u^N (t,x ) ) , \\
& \frac{ \partial }{ \partial y_2 (x)  } E_{ \nu_{ (u^N (t, x ) , v^N (t, x) ) + y (x) } } [ b (x) ]  \bigg|_{ y (x) = 0 } 
=  E_{\nu^N_t }  [ b (x) ( \eta_2 (x) - v^N (t, x) ) ] / \chi_2 (v^N (t,x ) )  .  
\end{aligned}
\]
Now the Taylor expansion around $y (x) = (0, 0) $ enables us to write 
\[
\begin{aligned}
E_{ \nu_{ (u^N (t ,x ), v^N (t , x ) ) + y (x)  } } [ b (x) ]  = 
& E_{ \nu_{ ( u^N (t,x ) , v^N (t, x ) ) }  } [ b (x ) ]  \\
& + \frac{\partial }{ \partial y_1 (x) } E_{ \nu_{ (u^N (t, x ) , v^N (t, x )  ) + y (x) } } [ b (x) ] \bigg|_{ y (x) = 0 } y_1 (x)  \\
& + \frac{\partial }{ \partial y_2 (x) } E_{ \nu_{ (u^N (t, x ) , v^N (t, x )  ) + y (x) } } [ b (x) ] \bigg|_{ y (x) = 0 } y_2 (x) 
+ C (\delta) | y (x)  |^2. 
\end{aligned}
\]
In the sequel, we estimate the expectations appearing in the last display to finish the proof. For the first term, since $R (x) $ has mean zero with respect to $\nu^N_t $ for every $x \in \mathbb{T}^d_N$, we can easily see that there exists a positive constant $C$ such that
\[
\begin{aligned}
| E_{ \nu_{ (u^N (t ,x ), v^N (t , x ) ) } } [b (x)]  | 
& = \big| - E_{ \nu_{ (u^N (t ,x ), v^N (t , x ) ) } } [ R (x) \mathbf{1}_{ \{ \sum_{ z \in \Lambda_h} \eta_1 (x + z )  > A  \} } ] \big|  \\
& \le A^{-1} \big|  E_{ \nu_{ (u^N (t ,x ), v^N (t , x ) ) } } \big[ R ( x ) \sum_{ z \in \Lambda_h } \eta_1 ( x + z )  \big] \big| 
\le C A^{ -1 } 
\end{aligned}
\]
where we used Markov's inequality in the penultimate estimate, and the last estimate can be deduced since moments of any order is bounded uniformly in $N $. On the other hand, due to a similar calculation given above, we have
\[
\begin{aligned}
& \bigg| \frac{ \partial }{ \partial y_1 (x) } E_{ \nu_{ (u^N (t, x ), v^N (t, x ) ) + y (x)  } } [ b (x) ] \bigg|_{ y = 0 }  \bigg| \\
& \quad = \bigg| \frac{ \partial }{ \partial y_1 (x) } E_{ \nu_{ (u^N (t, x ), v^N (t, x ) ) + y (x)  } } [ R (x) \mathbf{1}_{ \{  \sum_{ z \in \Lambda_h } \eta_1 (x + z )  > A \} } ] \bigg|_{ y = 0 } \bigg|  \\
& \quad \le A^{ - 1 } \chi_1(u^N (t,x ) )^{-1}  \bigg| E_{ \nu_{ (u^N (t,x ), v^N (t, x ) ) } } \big[ R (x) \sum_{ z \in \Lambda_h } \eta_1 (x + z ) ( \eta_1 (x) - u^N ( t,  x ) ) \big]  \bigg| .
\end{aligned}
\]
However, since $\chi_1 (u^N (t,x ) )^{-1} = \varphi^\prime (u^N (t, x ) ) / \varphi (u^N (t, x ) ) \le C e^{ C_1 K } $ in view of the uniform lower bound for $u^N $ by Lemma \ref{unif est}, the last display can be bounded by $C A^{-1 } e^{C_1 K } $. Similarly for type-$2$ configuration, we obtain 
\[
\bigg| \frac{ \partial }{ \partial y_1 (x) } E_{ \nu_{ (u^N (t, x ), v^N (t, x ) ) + y (x)  } } [ b (x) ] \bigg|_{ y = 0 }  \bigg|  
= O ( A^{-1} e^{ C_1 K } )
\]
and thus we complete the proof. 
\end{proof}

\begin{lemma}
\label{R2lem2}
There exists a positive constant $C$ such that for sufficiently small positive constants $\gamma_2 $ and $ \delta $ we have
\[
\sup_{  \ell \in  \mathbb{N}} 
E_{ \nu^N_t } [ e^{ \gamma_2 (2 \ell + 1 )^d | \tilde{ y } (x) |^2 } \mathbf{1}_{ \{  | \tilde{ y } (x) |  \le \delta  \} } ]  < C . 
\]
\end{lemma}
Since type-$2$ particles satisfy the exclusion rule, recalling $| \tilde{ y } (x) | = | \tilde{y }_1 (x) | + | \tilde{y}_2 (x)  | $ and $\mathbf{1}_{ \{ |\tilde{ y } (x)  \le \delta \} } \le \mathbf{1}_{ \{ |\tilde{ y }_1 (x)  \le \delta \} } $ for each $x \in \mathbb{T}^d_N$, it suffices to show the assertion only for type-$1$ configuration. However, such a result for single species model is already proved in \cite{EFHPS20} so that we omit the proof here. With these lemmas at hand, we can estimate $R_2 $ as follows. 

\begin{lemma}
\label{R2}
There exists a positive constant $C$ such that  
\[
\int_0^T E_{ \mu^N_t } \bigg[ \bigg|  \sum_{ x \in \mathbb{T}^d_N }  R_2 (x) \bigg| \bigg]  dt 
\le \int_0^T C H (\mu^N_t | \nu^N_t ) dt 
 +  C N^d \ell^{-d} + C\varepsilon^{-1} N^{d-2} \ell^{ 2 }  + C N^d A^{ -1 } e^{C_1 K } .
\]
\end{lemma}
\begin{proof}
According to Lemma \ref{R2lem1}, we have
\[
\begin{aligned}
| R_2 (x) | \le C | y (x) |^2 \mathbf{1}_{ \{ |y (x) | \le \delta  \} } +  C \ell^{-d }  +  C A ^{- 1 } e^{C_1 K } . 
\end{aligned}
\]
Here recall the definition of scaled local averages $y (x ) $ and $\tilde{y} (x) $ given at the beginning of this section. By an elementary inequality $(a + b )^2 \le 2 a^2 + 2 b^2 $ for every $a,b \in \mathbb{R}$, we have that 
\[
\begin{aligned}
& y_1 (x)^2 \le 2 \tilde{y}_1 (x)^2 + 2 \bigg( \frac{1}{ |\Lambda_\ell | } \sum_{  z \in \Lambda_{ \ell }  } (u^N (t, x +  z) - u^N (t, x) ) \bigg)^2,  \\
& y_2 (x)^2 \le 2 \tilde{y}_2 (x)^2 + 2 \bigg( \frac{1}{ |\Lambda_\ell | } \sum_{  z \in \Lambda_{ \ell }  } (v^N (t, x +  z) - v^N (t, x) ) \bigg)^2  \\
\end{aligned}
\]
for every $x \in \mathbb{T}^d_N$. Therefore, the energy estimate (Lemma \ref{energy est}) enables us to estimate 
\[
\begin{aligned}
 \int_0^T E_{ \mu^N_t } \bigg[ \sum_{ x \in \mathbb{T}^d_N} y (x)^2 \mathbf{1}_{ \{ | y (x) | \le \delta   \} } \bigg] dt  
\le C \varepsilon^{-1} N^{d-2 } \ell^{ 2 } 
+ 2 \int_0^T E_{ \mu^N_t } \bigg[ \sum_{ x \in \mathbb{T}^d_N } \tilde{y} (x)^2 \mathbf{1}_{ \{ | y (x) |\le  \delta  \} } \bigg] dt .  
\end{aligned}
\]
Now our task is to bound the second term in the last display. For that purpose, it is convenient to replace $y (x) $ in the indicator function $\mathbf{1}_{ \{ | y (x) | \le \delta \} }$ by $\tilde{y} (x) $ as 
\[
\begin{aligned}
 \mathbf{1}_{ \{ |y (x ) | \le \delta  \} } 
& = \mathbf{1}_{ \{ |y (x ) | \le \delta  \} }  
\big( \mathbf{1}_{ \{ | \tilde{y} (x ) | \le 2 \delta  \} } + \mathbf{1}_{ \{ | \tilde{y} (x ) | > 2 \delta  \} }  \big) \\
& \le \mathbf{1}_{ \{ | \tilde{y} (x ) | \le 2 \delta  \} } 
 + \mathbf{1}_{ \{ |y (x ) | \le \delta  \} } \mathbf{1}_{ \{ | \tilde{y} (x ) | > 2 \delta  \} } \\
& \le \mathbf{1}_{ \{ | \tilde{y} (x ) | \le 2 \delta  \} } 
 + \mathbf{1}_{ \{ |y (x ) | \le \delta  \} } 
\big( \mathbf{1}_{ \{ |y (x ) | > \delta  \} } + \mathbf{1}_{ \{ |y (x ) - \tilde{y} (x) | \ge \delta  \} }  \big) \\
& =  \mathbf{1}_{ \{ | \tilde{y} (x ) | \le 2 \delta  \} } 
+ \mathbf{1}_{ \{ |y (x ) | \le \delta  \} }  \mathbf{1}_{ \{ |y (x ) - \tilde{y} (x) | \ge \delta  \} } .
\end{aligned}
\]
In particular, we have
\[
\begin{aligned}
\tilde{y} (x)^2 \mathbf{1}_{ \{  | y (x) | \le \delta \} } \mathbf{1}_{ \{  | y (x) - \tilde{y} (x) | \ge \delta \} }  
& \le 2 \big(  y (x)^2 +   ( y (x) - \tilde{y} (x) )^2 \big)  \mathbf{1}_{ \{  | y (x) | \le \delta \} } \mathbf{1}_{ \{  | y (x) - \tilde{y} (x) | \ge \delta \} } \\
& \le 2 \delta^2  \mathbf{1}_{ \{  | y (x) - \tilde{y} (x) | \ge \delta \} }  + 2 ( y (x) - \tilde{y} (x) )^2  \\
& \le 4 (y(x) - \tilde{y} (x) )^2 
\end{aligned}
\]
where we used Markov's inequality for the last line. Therefore, according to the energy estimate given in Lemma \ref{energy est}, we deduce for some positive constant $C$ that 
\[
\int_0^T \sum_{ x \in \mathbb{T}^d_N } E_{ \mu^N_t } 
\big[ \tilde{y} (x)^2 \mathbf{1}_{ \{  | y (x) | \le \delta \} } \mathbf{1}_{ \{  | y (x) - \tilde{y} (x) | \ge \delta \} }  \big] dt  
\le C\varepsilon^{-1} N^{d-2} \ell^{ 2 } .
\]
On the other hand, applying the entropy inequality, we have
\[
\begin{aligned}
\sum_{ x \in \mathbb{T}^d_N  }E_{ \mu^N_t } [ \tilde{y} (x)^2 \mathbf{1}_{ \{ | \tilde{y} (x)| \le \delta \} } ] 
& \le \gamma^{ -1 } H (\mu^N_t | \nu^N_t ) + \gamma^{-1} \log E_{ \nu^N_t } [ e^{  \gamma \sum_{ x } \tilde{y}^2 (x) \mathbf{1}_{ \{ | \tilde{y} (x) | \le 2 \delta \} } } ] \\
& \le \gamma^{ -1 } H (\mu^N_t | \nu^N_t ) + \gamma^{-1} \ell^{-d} \sum_{x \in \mathbb{T}^d_N } \log E_{ \nu^N_t } [ e^{  \gamma \ell^d  \tilde{y} (x)^2 \mathbf{1}_{ \{ | \tilde{y} (x) | \le 2 \delta \} } } ] 
\end{aligned}
\]
for every positive constant $\gamma $. However, we have that
\[
\begin{aligned}
\log E_{\nu^N_t } [ e^{ \gamma \ell^d \tilde{y} (x)^2 \mathbf{1}_{ \{ |\tilde{y} (x) | \le 2 \delta \} } } ]
\le \log \big( 1 + E_{ \nu^N_t } [ e^{ \gamma \ell^d \tilde{y} (x)^2 } \mathbf{1}_{ \{ |\tilde{y} ( x ) | \le 2 \delta  \} } ] \big) 
\end{aligned}
\]
is bounded uniformly in $\ell \in \mathbb{N}$ by Lemma \ref{R2lem2} for small $\gamma$, $\delta$. Since $\gamma $ can be taken as a constant independent of $\ell$ and $N$, we obtain the assertion and complete the proof. 
\end{proof}

\subsubsection{Estimate of $R_3$}
Finally we are concerned with $R_3 $.

\begin{lemma}
\label{R3}
For a small positive constant $\gamma_3$, we have  
\[
\mathbb{E}^N \bigg[ \bigg| \int_0^T  \sum_{ x \in \mathbb{T}^d_N } R_3 (x) dt  \bigg|  \bigg] 
\le  \gamma_3^{-1} \int_0^T H ( \mu^N_t | \nu^N_t ) dt + \frac{ C \varepsilon^{-1} B N^{ d-2 } \ell^2 }{ \delta^2  }  + \frac{ C N^d }{ \gamma_3 \ell^d } e^{ -c_1 \ell^d }. 
\]
\end{lemma}
\begin{proof}
First one can easily see that there exists a positive constant $C = C ( |\Lambda_h | )$ such that 
\[
\begin{aligned}
E_{ \nu_\beta } \big[ | R (x) |  \mathbf{1}_{ \{  \sum_{ y \in \Lambda_h} \eta_1 (x + y ) \le A \} } | \eta_1^\ell (x), \eta_2^\ell (x) \big] 
& \le C \sum_{ y \in \Lambda_h } \big( g_1 (y, \eta_2^\ell (x) )  \eta_1^\ell (x) + g_2 (y,  \eta_2^\ell (x) ) \big) \\
& \le C \sum_{ y \in \Lambda_h } \big( g_1 (y, \eta_2^\ell (x) )  \tilde{ y }_1 (x) + g_2 (y,  \eta_2^\ell (x) ) \big) 
\end{aligned}
\]
where $g_i (y, \eta_2^\ell (x) ) $ ($ i = 1,2$) denote polynomials of $\eta_2^\ell (x) $ obtained by substituting $\eta_2^\ell (x) $ into all argument $ \{ \eta_2 (z) ; z \in \Lambda_h \} $ and in the second line we used the uniform boundedness of $ u^N $ and $ v^N $ (Lemma \ref{unif est}) to replace $\eta_1^\ell (x) $ by $\tilde{y}_1 (x ) $. Moreover, in view of Lemma \ref{unif est}, noting $g_1 (y) $ and $g_2(y) $ are polynomials of $\{ \eta_2 (z) ; z \in \Lambda_h \} $, which are uniformly bounded, we see that there exists a positive constant $C$ such that 
\[
\sup_{ x \in \mathbb{T}^d_N  } \big| \big( g_1 (y, \eta_2^\ell (x) ) \tilde{y}_1 (x)  + g_2 (y, \eta_2^\ell (x) )  \big) 
\mathbf{1}_{ \{ \eta_1^\ell (x) \le B \} } \big|
\le C ( B + 1) \le C B 
\]
for sufficiently large $N$. 

Now we decompose
\[
 y (x)  = \tilde{y} (x) 
 + \frac{1}{ | \Lambda_\ell | } \sum_{ z \in \Lambda_{\ell }  } (u^N (t, x + z) - u^N (t, x) )
+ \frac{1}{ | \Lambda_\ell |  } \sum_{ z \in \Lambda_{ \ell }  } (v^N (t, x + z) - v^N (t, x) ) .
\]
Then with the help of Markov's inequality, we have
\[
\begin{aligned} 
&\int_0^T \sum_{ x \in \mathbb{T}^d_N } E_{ \mu^N_t  } \bigg[ \sum_{ y \in \Lambda_h  } \big( g_1 (y, \eta_2^\ell (x) ) \tilde{y}_1 (x) + g_2 (y, \eta_2^\ell (x) ) \big)  
\mathbf{1}_{ \{ \eta_1^\ell (x) \le B \} } 
\mathbf{1}_{ \{ | y (x) |  > \delta \} } \bigg] dt  \\ 
&  \le \int_0^T \sum_{ x\in \mathbb{T}^d_N } E_{ \mu^N_t  } \bigg[ \sum_{y \in \Lambda_h } \big( g_1 (y, \eta_2^\ell (x) ) \tilde{y}_1 (x)  + g_2 (y, \eta_2^\ell (x) )\big)  
\mathbf{1}_{ \{ \eta_1^\ell (x) \le B \} } 
\mathbf{1}_{ \{ | \tilde{y} (x) |  > \delta / 2 \} } \bigg] dt  \\
& \quad +  \frac{ C B }{ \delta^2 } \int_0^T \sum_{ x \in \mathbb{T}^d_N }  
\bigg( \frac{1}{ (2\ell + 1 )^d } \sum_{ z \in \Lambda_{ \ell, x}  } ( u^N (t, z) - u^N (t, x) )  \bigg)^2 dt \\
& \quad + \frac{ C B }{ \delta^2 } \int_0^T \sum_{ x \in \mathbb{T}^d_N }  
\bigg( \frac{1}{ (2\ell + 1 )^d } \sum_{  z \in \Lambda_{\ell, x }  } ( v^N (t, z) - v^N (t, x) )  \bigg)^2 dt \\
& \le C (1 + 2\delta^{-1 } )  \int_0^T 
E_{ \mu^N_t } \bigg[  \sum_{x \in \mathbb{T}^d_N }  |\tilde{y} (x) | \mathbf{1}_{ \{ | \tilde{y} (x) | > \delta / 2  \} } \bigg] dt 
+ \frac{ C \varepsilon^{-1} B N^{d-2} \ell^2 }{ \delta^2 }
\end{aligned}
\]
where we used the above decomposition in the penultimate estimate, and for the last line we noted the uniform energy estimate given in Lemma \ref{energy est}. 

In the sequel, we estimate the first term in the last display in a similar way as in \cite{EFHPS20}. First, applying the entropy inequality, the integrand of the first term in the last line is bounded from above by 
\begin{equation}
\label{R3term}
\begin{aligned}
& \gamma_3^{-1} H ( \mu^N_t | \nu^N_t )  \\
& \quad + \frac{1}{ \gamma_3 ( 2 \ell + 1 )^d } \sum_{ x \in \mathbb{T}^d_N } \log \bigg( 1 - \nu^N_t ( |\tilde{y} (x) | >  \delta/ 2 ) 
+ E_{ \nu^N_t } [ e^{ \gamma_3 (2 \ell + 1 )^d  |\tilde{y} (x) | } \mathbf{1}_{\{ | \tilde{y} (x) | > \delta / 2  \} } ]  \bigg) . 
\end{aligned}
\end{equation}
By the Schwarz inequality, we have
\[
 E_{ \nu^N_t } \big[ e^{ \gamma_3 (2 \ell + 1 )^d  | \tilde{y} (x) | } \mathbf{1}_{\{ | \tilde{y} (x) | > \delta / 2  \} } \big]  
\le  E_{ \nu^N_t } \big[ e^{ 2 \gamma_3 (2 \ell + 1 )^d  |\tilde{y} (x) | } \big]^{1/2} \cdot
\nu^N_t ( | \tilde{y} (x) | > \delta / 2 )^{1/2} .
\]
Now we fix an arbitrary positive constant $s$. Then by an elementary inequality $e^{ |x| } \le e^{-x} + e^x $ and Markov's inequality we have
\[
\begin{aligned}
\nu^N_t ( | \tilde{y}_1 (x) | > \delta )   
& \le E_{ \nu^N_t } [ e^{ s \tilde{y}_1 (x) \ell^d } ] e^{ - s \ell^d \delta } 
+ E_{ \nu^N_t } [ e^{ - s \tilde{y}_1 (x) \ell^d } ] e^{ - s \ell^d \delta }  \\
& \le \bigg( \prod_{ |z| \le \ell } E_{ \nu^N_t } [ e^{ s (\eta_1 (x+z ) - u (x + z ) ) } ] 
+ \prod_{ |z| \le \ell } E_{ \nu^N_t } [ e^{ - s (\eta_1 (x+z ) - u (x + z ) ) } ]  \bigg) e^{ - s \ell^d \delta } .
\end{aligned}
\]
Moreover, since the reference measure $\nu^N_t $ has any finite moment, we can expand 
\[
\log E_{ \nu^N_t } [ e^{ \pm s (\eta_1 (x) - u (x) ) }] = s^2 \sigma_1^2 /2 + o (s^2 )
\]
as $s \to 0$. Note here that $\sigma_1^2 $ is bounded uniformly in $N$ in view of Lemma \ref{unif est}. Therefore, by taking $s = \varepsilon \delta $ with $\varepsilon > 0$ small, we have 
\[
\nu^N_t ( |\tilde{y}_1 (x) | > \delta) 
\le 2 \prod_{ |z| \le \ell } e^{ \delta^2 ( \varepsilon - \varepsilon^2 \sigma_1^2 / 2 )  } 
\le e^{- c \ell^d }
\]
for some positive constant $c = c (\delta, \varepsilon)$. The same tail estimate of $\tilde{y}_2$ holds and thus we have $\nu^N_t ( | \tilde{y} (x ) | > \delta / 2 ) \le 2 e^{ - c \ell^d} $ for some positive constant $c$.

On the other hand, noting the uniform boundedness of $u^N $ by Lemma \ref{unif est} to see that 
\[
\begin{aligned}
E_{ \nu^N_t } [ e^{ 2 \gamma_3  ( 2 \ell + 1 )^d  | \tilde{y} (x) |  }  ]  
& = \prod_{ z  \in \Lambda_{ \ell, x }  } 
E_{ \nu^N_t } [ e^{ 2 \gamma_3 | \eta_1 ( z) - u^N (t,  z ) | } ]  
E_{ \nu^N_t } [ e^{ 2 \gamma_3 | \eta_2 ( z) - v^N (t,  z ) |  }  ]  \\
& \le C e^{ 2 \gamma_3 ( M_u \vee M_v ) (2  \ell + 1 )^d } 
 \prod_{ z  \in \Lambda_{ \ell, x }  } E_{ \nu^N_t } [ e^{ 2 \gamma_3  \eta_1 ( z)   }  ]  . 
\end{aligned}
\]
for every $x \in \mathbb{T}^d_N $. Since the zero-range jump rate $g$ satisfies a sub-linear growth by the assumption \textbf{(LG)} so that the partition function can be bounded as 
\[
Z_{u^N (t, x ) } = \sum_{ k \ge 0 } \frac{ \varphi (u^N (t, x ) )^k }{ g (k) ! } \ge \sum_{ k \ge 0 } \frac{ (\varphi (M_u )^k }{  c^k  k ! } = e^{ \varphi (M_u ) / C } \ge 1.
\]
Therefore, we can bound the exponential moment for type-$1$ configuration as 
\[
E_{\nu^N_t } [ e^{ 2 \gamma_3 \eta_1 ( z )  } ] 
=  \frac{ 1}{ Z_{ u^N (t, z) } } \sum_{ k \ge 0 } e^{ 2 \gamma_3 k }  \frac{ \varphi (u^N (t, z ) )^k }{ g (k ) ! } 
\le  \sum_{ k \ge 0 } \frac{ ( C e^{ 2 \gamma_3  } \varphi (M_u ) )^k }{  k ! } 
\le C, 
\]
which stays finite for every $z \in \Lambda_{ \ell, x } $. Here in the penultimate estimate we used the bound $g (k ) \ge C k$ due to the assumption \textbf{(LG)}. Combining all estimates obtained in the above, now we can finish the proof. Indeed, using an elementary inequality $\log (1 + x ) \le x $ for $x \ge 0$ and taking $\gamma_3 $ sufficiently small, \eqref{R3term} can be bounded above by
\[
\gamma_3^{-1 } H ( \mu^N_t | \nu^N_t ) +  C \gamma_3^{ - 1 } \ell^{-d } N^d  e^{ - c_1 \ell^d } 
\]
for some positive constant $c_1 = c_1 (\delta, \varepsilon, \gamma_3 )$ and thus we complete the proof. 
\end{proof}

\section{Convergence results for semi-discretized system \eqref{dHDL eq}}
\label{sec:PDE}
In this section we give a proof of Theorem \ref{PDE thm}. Throughout this section, let $u^N  $ and $v^N $ be a solution of the system \eqref{dHDL eq} with initial functions satisfying the assumptions (A\ref{IF}) and (A\ref{PDE}) which are extended on $\mathbb{T}^d $ by \eqref{extension}. First we show the uniqueness of the problem \eqref{wStefan} in a similar way as \cite{HHP00}. 

\begin{lemma}
\label{uniqueness}
The problem \eqref{wStefan} has at most one solution. 
\end{lemma}
\begin{proof}
Let $w_1, w_2 \in L^\infty (Q_T) $ be two solutions of the problem \eqref{wStefan} with a common initial function. Subtracting the weak forms of $w_1$ and $w_2$ with each other, we have 
\begin{equation}
\label{subtracted weak forms}
- \iint_{ Q_T } (w_1 - w_2 ) \psi_t d\theta dt 
+ \iint_{ Q_T } \big( \nabla \mathcal{D}_\varphi (w_{1} ) - \nabla \mathcal{D}_\varphi (w_{2 } ) \big) \cdot \nabla \psi d \theta dt = 0 
\end{equation}
where the function $ \mathcal{D}_\varphi $ is defined by $\mathcal{D}_\varphi (s) = \varphi (s) \mathbf{1}_{ [ 0, \infty ) } $, which is non-decreasing on $\mathbb{R}$. Now we take a test function $\psi $ defined by
\[
\psi (t, \theta ) = \int_t^T \big( \mathcal{D}_\varphi (w_{1} ) - \mathcal{D}_\varphi (w_{2 } ) \big) (\tau, \theta)  d \tau 
\]
in \eqref{subtracted weak forms}. Then, noting a simple relation $f (t ) \cdot \int_t^T f (\tau)  d \tau = -\frac{1}{2} \frac{d}{dt} \big|  \int_t^T f (\tau) d\tau  \big|^2$ for any vector-valued continuous function $f$, the identity \eqref{subtracted weak forms} becomes 
\[
\begin{aligned}
 \iint_{Q_T }  (w_1 - w_2 ) ( \mathcal{D}_\varphi (w_1 ) - \mathcal{D}_\varphi (w_2 ) ) dt d \theta 
+  \frac{1}{2}  \int_{\mathbb{T}^d } \bigg| \int_0^T \nabla \big( \mathcal{D}_\varphi (w_{1} ) - \mathcal{D}_\varphi (w_{2 } ) \big) dt  \bigg|^2 d\theta = 0 .  
\end{aligned}
\]
Here, the first term of the left-hand side of the last display stays non-negative since the function $\mathcal{D}_\varphi $ is non-decreasing on $\mathbb{R}$ so that it becomes to be zero.  Hence we have $w_1 = w_2 $ and complete the proof. 
\end{proof}

Next we show the compactness of the discretized solutions $u^N$ and $v^N$.  

\begin{lemma}
\label{u relcpt}
The sequence $\{ u^N (t, \theta) \}_{ N \in \mathbb{N} }$ is relatively compact in $L^p (Q_T) $ for any $p \ge 2$. 
\end{lemma}
\begin{proof}
By the Fr\'{e}chet-Kolmogorov theorem (\cite{B83}, Theorem \Rnum{4}.25 and Corollary \Rnum{4}.26), it suffices to show that there exists a positive constant $C$ such that 
\begin{align*}
\int_0^{T-\tau} \int_{\mathbb{T}^d} | u^N(t+\tau,\theta) -u^N(t,\theta) |^p d\theta dt  \leq C \tau, \\
\int_0^T \int_{\mathbb{T}^d} |u^N(t,\theta+\alpha) - u^N(t,\theta) |^p d\theta dt \leq C | \alpha |
\end{align*}
for all $p \ge 2$, $\tau \in (0,T)$ and $\alpha \in \mathbb{R}^d$ sufficiently small. 

First we show the equi-continuity along spatial direction with exponent $p= 1$. Once the case when $p =1$ is proved, then we obtain the assertion for any exponent $p \ge 1$ according to the uniform boundedness of $u^N$ (Lemma \ref{unif est}). Change of variables enables us to restrict our cases for non-negative $\alpha$. In this case, we observe 
\[
\begin{aligned}
& \iint_{Q_T} \big| u^N (t, \theta + \frac{ n}{N} ) - u^N (t, \theta) \big| d\theta dt \le \frac{ n }{N} \iint_{ Q_T } | \nabla^N u^N (t, \theta) | d\theta dt , \\
& \iint_{ Q_T } \big| u^N (t, \theta + \frac{ 1}{ r N }) - u^N (t, \theta) \big| d\theta dt \le \frac{1 }{r N } \iint_{ Q_T} | \nabla^N u^N (t, \theta) | d\theta dt  
\end{aligned}
\]
for every $t \in [0,T]$, $n \in \mathbb{Z}_+ $ and $r \ge 1$. Combining these two estimates and applying them for $\alpha = n / rN$ with $n = \lceil \alpha N \rceil$ and $r = \lceil \alpha N \rceil  / \alpha N $ to obtain
\[
\iint_{Q_T} | u^N (t, \theta + \alpha ) - u^N (t, \theta) | d\theta dt \le \alpha \iint_{Q_T } | \nabla^N u^N (t, \theta) | d\theta dt \le \alpha \| \nabla^N u^N \|_{ L^2 (Q_T)}
\]
for every $t \in [0,T]$ and $\alpha \ge 0$ where in the last estimate we used H\"{o}lder's inequality. According to the uniform energy estimate Lemma \ref{energy est}, we obtain the equi-continuity in spatial variables for any index $p \ge 1$. In particular, the second assertion holds for any $p \ge 2$. 

Similarly, we next prove the equi-continuity in time only for the case $p=2$. We remark here that when $1 \le p < 2$ another exponent for $\tau$ is needed so that we restrict our cases only for $p \ge 2$. The integral appearing in the left hand side of the first estimate for $p=2$ is equal to   
\[
\int_0^{T-\tau} \int_{\mathbb{T}^d} \left(  \int_0^\tau  \partial_t u^N(t+s,\theta)  ds \right)  \left( u^N(t+\tau,\theta) -u^N(t,\theta) \right)  d\theta dt 
.\]
However, using the first equation of (\ref{dHDL eq}) for the integrand and integrating by parts, this quantity can be estimated from above by
\begin{align*}
&\int_0^\tau
 \left( \int_0^{T-\tau} \int_{\mathbb{T}^d} \left| \nabla^N \varphi ( u^N(t+s,\theta) )  \right|^2 d\theta dt \right)^{ 1 /2 }
 \left( \int_0^{T-\tau} \int_{\mathbb{T}^d} \left| \nabla^N u^N ( t + \tau ,\theta) \right|^2 d\theta dt \right)^{ 1/2 } ds \\
&+\int_0^\tau
 \left( \int_0^{T-\tau} \int_{\mathbb{T}^d} \left| \nabla^N \varphi ( u^N(t+s,\theta) ) \right|^2 d\theta dt \right)^{1/2}
 \left( \int_0^{T-\tau} \int_{\mathbb{T}^d} \left| \nabla^N u^N(t,\theta) \right|^2 d\theta dt \right)^{1/2} ds \\
&+2K \int_0^\tau \int_0^{T-\tau} \int_{\mathbb{T}^d}  u^N (t+s,x) v^N (t+s,x) d\theta dt ds  
.\end{align*}
where we used Schwarz's inequality to estimate the first and the second terms. Furthermore, note here that the mean-value theorem enables us to write $\partial^N_j \varphi (u^N (t, x ) ) = \varphi^\prime (\tilde{u}_j ) \partial^N_j u^N (t , x )  $ with some $\tilde{u}_j $ between $u^N ( t,  x + e_j) $ and $u^N (t, x ) $ for every $j = 1, \ldots , d$, $t \in [0, T] $ and $ x \in \mathbb{T}^d_N$. Since $u^N$ takes values in a finite interval, the first and the second terms in the above display can be estimated by the $L^2$-energy of $\nabla^N u^N $. The third term can also be estimated by $\tau$ in view of Lemma \ref{react int}. By this line, we could bound the above sum of three terms by $\tau $ with some positive constant and the desired estimate was proved. 
\end{proof}

Moreover, we have the following weak compactness of the sequence $\{ v^N \}$. The next result is obvious since any bounded sequence of a reflective Banach space has a weakly convergent subsequence and the functions $v^N $ have a uniform bound in $N$ and so do their $L^p$-norms for $p >  1 $.

\begin{lemma}
\label{v relcpt}
The sequence $\{ v^N (t, \theta) \}_{ N \in \mathbb{N} }$ is weakly pre-compact in $L^p (Q_T) $ for any $p > 1$. Namely, there exists a subsequence $(N_k)$ and $v \in L^p (Q_T ) $ such that $v^N \rightharpoonup v$ weakly in $L^p (Q_T)$.  
\end{lemma}

Now we give a proof of Theorem \ref{PDE thm}.

\begin{proof}[Proof of Theorem \ref{PDE thm}]
For any $p >1$, by Lemma \ref{u relcpt} the sequence $\{ u^N(t,\theta) \}_{N \in \mathbb{N}}$ is strongly precompact in $L^p(Q_T)$, while by Lemma \ref{v relcpt} $\{ v^N(t,\theta) \}_{N \in \mathbb{N}}$ is weakly precompact in $L^p (Q_T)$. Therefore, there exist a subsequence $\{  N_k \}$ and functions $u, v \in L^p (Q_T)$ such that 
\[
\begin{aligned}
u^{N_k} \to u \text{ strongly in } L^p(Q_T), \quad v^{N_k} \rightharpoonup v \text{ weakly in } L^p(Q_T)
\end{aligned}
\]
for any $p > 1$. In particular, by taking further subsequences if necessary (which again denoted by $N_k$), we see that $u^{N_k} \to u$ a.e. in $Q_T$, which clearly take values in $[0, M_u ] \times [0, M_v ] $. We show the function $u  $ belongs to $L^2 (0, T; H^1 (\mathbb{T}^d))$. For any test function $\psi \in C^\infty (\mathbb{T}^d)$, $j=1,\ldots , d $ and $t \in [0, T]$, we have 
\[
\int_{ \mathbb{T}^d } u^N (t, \theta) \partial^N_j \psi (\theta ) d\theta = - \int_{ \mathbb{T}^d } \psi (\theta ) \partial^N_j u^N (t, \theta)  d\theta 
\]
where $\partial^N_j $ is the discrete partial derivative on $j$-th direction defined by $\partial^N_j G (\theta) = N [ G (\theta + \frac{ e_j}{N}) - G (\theta) ] $ for every continuous functions $G$ on $ \mathbb{T}^d $. Taking limit along $(N_k)$ on the above identity, we see that $\partial^N_j u^N $ converges to the $j$-th partial derivative $\partial_j u$ in distributional sense for every $j=1,\ldots ,d$. Moreover, since $L^2 (\mathbb{T}^d )$-norm of the discrete derivative $\partial^N_j u^N (t, \cdot )$ is bounded above by some constant independent of $N$ in view of Lemma \ref{energy est}, $\partial_j u (t, \cdot)$ belongs to $L^2 (\mathbb{T}^d)$ for every $j = 1,\ldots ,d$ and thus we obtain $u \in L^2 (0,T; H^1 (\mathbb{T}^d))$. Moreover, by the second equation of (\ref{dHDL eq}), we have
\[
\iint_{Q_T} u^N(t,\theta) v^N(t,\theta) d\theta dt \le \frac{1}{K }
\] 
for every $N \in \mathbb{N}$. Since $u^{N_k} \to u$ strongly in $L^2 (Q_T)$ and $v^{N_k} \rightharpoonup v $ weakly in $L^2(Q_T)$ as $k$ tends to infinity, their product $u^{N_k}v^{N_k}$ converges strongly in $L^1(Q_T)$ to $u v$. Therefore, taking limit along $(N_k )$ on the above bound, we get $uv=0$ a.e. in $Q_T$.

Next we let $w^N \coloneqq u^N - v^N $ and we denote its extension as a simple function on $\mathbb{T}^d$ (see \eqref{extension}) by the same notation. Note here that it is already shown that the sequence $(w^N)$ on $Q_T$ converges weakly in $L^2 (Q_T ) $ to some $w$ along the subsequence $( N_k )$. We show any limit point $w$ satisfies (\ref{weakform}). For that purpose, we subtract the two equations in \eqref{dHDL eq} with each other to cancel singular reaction terms. Let $\psi $ be an arbitrary smooth function with compact support on $Q_T$ satisfying $\psi (T, \cdot) \equiv 0$ on $\mathbb{T}^d$. Then, for every $t \in [0, T] $ and $\theta = ( \theta_j )_{ j = 1,\ldots ,d } \in \mathbb{T}^d $, we multiply $\psi (t, \theta ) \prod_{ j = 1}^d \mathbf{1}_{ [\frac{ x_j }{ N } - \frac{ 1}{ 2N } , \frac{ x_j }{ N } + \frac{ 1}{ 2N} ) } (\theta_j ) $ and sum up over $x \in \mathbb{T}^d_N$ to obtain 
\[
\partial_t w^N (t, \theta ) \psi (t, \theta) = \sum_{ x \in \mathbb{T}^d_N } \Delta^N \big( \varphi (u^N (t, x) ) -\varepsilon ( N ) v^N (t, x) \big)  \prod_{ 1 \le j \le d } \mathbf{1}_{ \big[ \frac{x_j}{ N}  - \frac{1}{ 2 N} , \frac{x_j}{  N} + \frac{1}{ 2N } \big) } (\theta_j ) \psi (t, \theta ) .
\]
Then we integrate over $(t, \theta ) \in Q_T $ to get an identity  
\begin{equation}
\label{weak dStefan}
\begin{aligned}
&  \int_{ \mathbb{T}^d } w^N (0, \theta ) \psi (0, \theta ) d\theta - \iint_{Q_T}  w^N (t, \theta ) \psi_t (t, \theta)  d\theta dt  \\
& \quad =  \iint_{ Q_T } \bigg( \sum_{ x \in \mathbb{T}^d_N } \varphi (u^N (t, x )  ) \prod_{1 \le j \le d } \mathbf{1}_{ \big[ \frac{x_j}{ N}  - \frac{1}{ 2 N} , \frac{x_j}{  N} + \frac{1}{ 2N } \big) } (\theta_j )  \bigg) \Delta^N \psi (t, \theta ) d\theta dt \\
& \qquad + \iint_{Q_T } \varepsilon (N) \nabla^N v^N (t, \theta ) \cdot \nabla^N \psi (t, \theta ) d\theta dt 
\end{aligned}
\end{equation}
due to the integration by parts formula where we defined $\nabla^N G (\theta) \coloneqq \big( N (G (\theta + \frac{ e_j}{ N} ) - G (\theta ) ) \big)_{ j = 1, \ldots , d } $ and $\Delta^N G (\theta) \coloneqq N^2 \sum_{ j = 1, \ldots , d } \big( G (\theta + \frac{ e_j}{ N} ) ) + G  (\theta - \frac{ e_j }{N } ) - 2 G ( \theta ) \big) $ for every $ G \in C (\mathbb{T}^d ) $.

Recall the definition of $u^N (t , \theta ) $ given in \eqref{extension}. Then, one can easily notice that 
\[
 \varphi (u^N (t, \theta ) ) = \sum_{ x \in \mathbb{T}^d_N } \varphi (u^N (t, x )  ) \prod_{ 1 \le j \le d} \mathbf{1}_{ \big[ \frac{x_j}{ N}  - \frac{1}{ 2 N} , \frac{x_j}{  N} + \frac{1}{ 2N } \big) } (\theta_j ) 
\] 
since the indicator functions are disjoint. On the other hand, the second term in the right-hand side of \eqref{weak dStefan} is absolutely bounded by 
\[
\bigg( \iint_{ Q_T } \varepsilon (N) | \nabla^N v^N |^2 d\theta dt  \bigg)^{1/2} 
\bigg( \iint_{ Q_T } \varepsilon (N) | \nabla^N \psi |^2 d\theta dt  \bigg)^{1/2} ,  
\]
which vanishes as $ N $ tends to infinity due to the energy estimate (Lemma \ref{energy est}). Now recalling the functions $\{ u^N (t ,\theta ) \}_{ N \in \mathbb{N} } $ is relatively compact according to Lemma \ref{u relcpt}, we take the limit of scaling parameters along the subsequence $(N_k ) $. Then we obtain the desired weak form (\ref{weakform}) for every $\psi \in H^1 (Q_T ) $ such that $\psi (T, \cdot ) \equiv 0 $ noting the space of smooth functions on $Q_T$ with compact support is dense in $H^1 (Q_T) $. Finally, the uniqueness of the problem \eqref{wStefan} assures that the  above convergence holds without taking any subsequence and thus we complete the proof.

\end{proof}

\section*{Acknowledgments}
The author would like to thank Professor T. Funaki and Professor M. Sasada for helping him with variable suggestions.


\end{document}